\documentclass[3p]{elsarticle} 

\usepackage{amsmath}
\usepackage{amssymb}
\usepackage{amsthm}
\usepackage{natbib}
\usepackage{graphicx}
\usepackage{subcaption}
\usepackage{amsfonts}
\usepackage[mathscr]{eucal}
\usepackage{xurl}
\usepackage{booktabs}
\usepackage{xcolor}
\usepackage{algorithm}
\usepackage{algorithmic}
\usepackage{ifthen}
\usepackage{caption}
\usepackage{nicefrac}

\theoremstyle{plain}
\newtheorem{theorem}{Theorem}[section]

\newtheorem{lemma}[theorem]{Lemma}
\newtheorem{corollary}[theorem]{Corollary}
\theoremstyle{definition}

\theoremstyle{remark}
\newtheorem{remark}[theorem]{Remark}

\def\pP{{\mathbb P}}
\def\eE{{\mathbb E}}

\def\nN{{\mathbb N}}
\allowdisplaybreaks[4]

\newcommand{\R}{\mathbb{R}} \newcommand{\N}{\mathbb{N}}
 \newcommand{\K}{\mathcal{K}}
\renewcommand{\P}{\mathcal{P}} \renewcommand{\S}{\mathcal{S}}
\newcommand{\D}{\mathcal{D}} \newcommand{\ra}{\rightarrow}
 \newcommand{\da}{\downarrow}
 
\newcommand{\bx}{\boldsymbol{x}} \newcommand{\bp}{\boldsymbol{p}}
\newcommand{\bd}{\boldsymbol{d}}
\newcommand{\bsigma}{\boldsymbol{\sigma}}
\DeclareMathOperator*{\argmax}{arg\,max}

\def\nialgo{NI-ME}
\def\ibalgo{IB-CME}

\newboolean{showcomments}
\setboolean{showcomments}{true}
\newcommand{\jk}[1]{\ifthenelse{\boolean{showcomments}} {\textcolor{red}{(JN says: #1)}} {} }
\newcommand{\nk}[1]{\ifthenelse{\boolean{showcomments}} {\textcolor{red}{(NK says: #1)}} {} }
\newcommand{\MP}[1]{\ifthenelse{\boolean{showcomments}} {\textcolor{red}{(MP says: #1)}} {} }

\newcommand{\ignore}[1]{}

\usepackage{times}
\title{Fixed confidence community mode estimation}

\begin{document}

\author[inst1]{Meera Pai}

\affiliation[inst1]{organization={Department of Electrical Engineering},
            addressline={IIT Bombay}, 
            country={India}
            }

\author[inst1]{Nikhil Karamchandani}
\author[inst1]{Jayakrishnan Nair}

\begin{abstract}%
  Our aim is to estimate the largest community (a.k.a., mode) in a
  population composed of multiple disjoint communities. This
  estimation is performed in a fixed confidence setting via sequential
  sampling of individuals with replacement. We consider two sampling
  models: (i) an identityless model, wherein only the community of
  each sampled individual is revealed, and (ii) an identity-based
  model, wherein the learner is able to discern whether or not each
  sampled individual has been sampled before, in addition to the
  community of that individual. The former model corresponds to the
  classical problem of identifying the mode of a discrete
  distribution, whereas the latter seeks to capture the utility of
  identity information in mode estimation. For each of these models,
  we establish information theoretic lower bounds on the expected
  number of samples needed to meet the prescribed confidence level,
  and propose sound algorithms with a sample complexity that is
  provably asymptotically optimal. Our analysis highlights that
  identity information can indeed be utilized to improve the
  efficiency of community mode estimation.%
\end{abstract}
	
\footnote{Nikhil Karamchandani and Jayakrishnan Nair acknowledge support from SERB via grant CRG/2021/002923 and two MATRICS grants.}
\maketitle

\section{Introduction}
\label{sec:intro}

There are several interesting applications which are based on
sequentially sampling individuals from an underlying
population. Examples include \textit{online cardinality estimation}
\citep{finkelstein1998, budianu2006sensors, Bres15, massoulie2006peer}
where the goal is to approximate the total size of the population and
algorithms are typically based on the counting of `collisions', i.e.,
instances where the sampled individual was `seen' before;
\textit{community exploration} \citep{chen2018community, liu2021multi,
  Bubeck2013} where the goal is to discover as many distinct entities
of interest as possible; and \textit{community detection / clustering}
\citep{yun2014community, chen2019proportionally,
  mazumdar2017clustering} where the underlying population is naturally
divided into communities / clusters and the goal is to estimate the
true clustering using samples from the population. 

In this work, we study a variant of the above problems where the
learning agent can sequentially sample individuals uniformly at random
from a population and the goal is to identify the largest community,
a.k.a. the \textit{community mode}. Some applications which motivate
our formulation are \textit{(i)} election polling where the community
of an individual is the candidate / party it supports and estimating
the community mode corresponds to predicting the winning candidate /
party; \textit{(ii)} identifying the dominant strain of a virus in an
infected population where the communities refer to different strains;
and \textit{(iii)} market research surveys where the goal is to
identify the most prevalent opinion in a population on some topic of
interest.

We consider two sampling models in our formulation, which differ in
the information revealed to the agent about each sample: \textit{(i)
  identityless sampling} where only the community that the sampled
individual belongs to is provided to the learning agent; and
\textit{(ii) identity-based sampling} where in addition to the
community information, the agent is also able to infer if the sampled
individual was seen before during the sampling process.

Since at each time, an individual is sampled uniformly at random from
the population, the probability that it belongs to a certain community
is proportional to the community size. Thus, under the identityless
sampling model, where the observation at each time is just the
community label, identifying the largest community becomes equivalent
to identifying the mode of an underlying discrete distribution from
independent and identically distributed (i.i.d.) samples. Given a
target error probability, the goal is to define a stopping rule which
minimizes the number of samples needed to estimate the mode
reliably. This problem was recently studied in
\cite{shah2020sequential} where an instance-dependent (i.e., it is a
function of the underlying probability distribution) lower bound on
the optimal sample complexity was derived and an algorithm based on
the idea of data-driven confidence intervals
\citep{maurer2009empirical} was proposed, for which the sample
complexity was shown to be within a constant factor of the
optimal. More recently, \cite{Jain2022} proposed an alternate stopping
rule based on prior-posterior-ratio (PPR) martingale confidence
sequences (see~\citep{waudby2020confidence}) and showed that its
sample complexity is in fact asymptotically optimal as the target
error probability goes to zero. For this problem, we consider an
alternate Chernoff-style stopping rule based on a novel Generalized
Likelihood Ratio (GLR) statistic \citep{chernoff1959sequential} and
prove its asymptotic optimality. (Interestingly, we discover a
connection between our GLR based algorithms and PPR based algorithms;
see Remark~\ref{remark:ppr}).

In contrast to the identityless sampling model, the identity-based
sampling model allows the learning agent to also know whether the
sampled individual at each time was seen before or not. A key focus of
this work is to study the impact of this additional `identity'
information on the sample complexity of community mode estimation.
The main challenge in designing and analyzing a sample-efficient
stopping rule for this setting is that the observations are no longer
i.i.d. across time; in fact the process has a Markovian structure. 
While a direct application of our GLR statistic for the identityless 
setting is computationally intractable in the identity-based setting, 
we devise a computationally tractable relaxation of this statistic that 
preserves both soundness as well as asymptotic optimality.
We clearly demonstrate, via both theoretical
results as well as numerical simulations, that there can be a
significant gap between the optimal sample complexity for community
mode estimation under the identityless and the identity-based sampling
models.

The GLR based test has found application in several other online
learning settings, including change point detection
\citep{siegmund1995using} and best arm identification in multi armed
bandits \citep{garivier2016optimal}. As mentioned before, one of the
key technical challenges for our setting is that the observations are
non i.i.d.; in fact the number of distinct individuals seen from each
community evolves as an absorbing Markov chain in the community mode
estimation problem.

The mode estimation problem has been studied classically in the batch
or non-sequential setting since many decades, see for example
\cite{parzen1962estimation, chernoff1964estimation}, with focus mainly
on proving statistical properties such as consistency for various
estimators. Another related problem with a long and rich history is
property estimation from samples generated from an underlying
distribution. Recent work has characterized the optimal sample
complexity for various properties of interest such as entropy
\citep{caferov2015optimal, acharya2016estimating}, support size and
coverage \citep{hao2019data, wu2018sample}, amongst others. Finally,
the problem of community mode estimation has been recently studied in
\cite{jain2021sequential} under the \textit{fixed budget} setting,
i.e., the number of samples is fixed and the goal is to design
decision rules which minimize the probability of error. In contrast,
our work focuses on the \textit{fixed confidence} setting where a
target error probability is specified but the number of samples is not
fixed apriori and the goal is to design stopping rules with small
sample complexity.

In this paper, we considered the learning task of online community
mode estimation in a fixed confidence setting, focusing on the role of
identity information. Our analysis relies heavily on the existing
GLR-based machinery for establishing correctness and asymptotic
optimality of PAC algorithms (most notably, the methodology developed
in~\cite{garivier2016optimal,karthik2020}). Our key novelty is: \\ (i)
From a modeling standpoint, we highlight the value of \emph{identity}
information in mode estimation.\\ (ii) We modify the GLR statistic
proposed in \cite{karthik2020} to design a tractable algorithm for the
identityless sampling model. A superior `1v1' variant of this
algorithm is also designed, which turns out to be identical to the
state of the art algorithm for (identityless) mode estimation--PPR-1v1
(see \cite{Jain2022}), which was designed using a different design
recipe.\\ (ii) We design a sound stopping condition under a
\emph{Markovian} observation sequence (arising under identity-based
sampling), and further design a computationally tractable relaxation
of the stopping statistic that preserves both soundness and asymptotic
optimality (Lemma~\ref{lemma:sup_pf_ST_Yabt_wi} is crucial to this
relaxation); indeed, that the discrete likelihood optimization over
$\mathbb{N}^K$ in~\eqref{def_Wabt} can be tractably relaxed is far from trivial.

\section{Problem formulation}


We consider a population of $N$ individuals. This population is
partitioned into~$K$ communities, the set of communities being denoted
by $\K = \{1,2,\ldots, K\}.$ In particular, community~$j$ has $d_j$
individuals, so that $\sum_{j\in \mathcal{K}} d_j = N.$ Our goal is to
identify the largest community (a.k.a., the community mode) in this
population, via sequential (random) sampling of individuals \emph{with
  replacement}, in a fixed confidence setting. At each sampling epoch,
an individual, selected uniformly at random from the population, is
`seen' by the learning agent. We make the assumption
  that the number of communities~$K$ is known to the learning
  agent. However, \emph{the total population size $N$ is assumed to be
    a priori unknown to the learning agent}. We consider two
different sampling models in this work, which differ in the
information that the learning agent obtains in each sampling epoch.

\ignore{
A mixed community setting with $K$ communities, i.e. a population
where all the $K$ communities are contained in a single partition(or
box) is considered. Individuals are drawn at randon from this box
sequentially with replacement. The size of community $j$ is denoted by
$d_j$ for $j \in \mathcal{K} = {1,2,\ldots K}$. Without loss of
generality we can assume that $d_1 > d_2 \geq \ldots d_{K-1} \geq d_K
$. Let $N = \sum_{j=1}^{K}d_j$. Let $p_i = \frac{d_j}{N}$ be the
probability of picking a sample from community $i$ where $i \in
\mathcal{K}$. Since the communities are assumed to be ordered, $p_1 >
p_2 \geq \ldots p_{K-1} \geq p_K $. Our goal is to estimate the
largest community or mode for this set of communities. We consider two
different sampling models given below for mode estimation.
}

\begin{enumerate}
\item \textit{Identityless sampling}:\\ Under this sampling model, at
  each time (a.k.a., sampling epoch) $t$, only the community
  associated with the sampled individual is revealed to the learning
  agent. In particular, let $x_t$ denote the community of the individual
  sampled at time $t$. Then the vector $\bx^t = (x_1, x_2, \ldots,
  x_t)$ represents the observations recorded by the learning agent
  till time $t$.

  Under the identityless sampling model, the agent is essentially
  recording i.i.d. samples from a distribution with support~$\K,$ with
  a probability mass ~$p_j = d_j / N$ on community~$j.$ Thus,
  community mode estimation boils down to identifying the mode of this
  distribution using i.i.d.  samples; this estimation has been
  analysed recently (also in the fixed confidence setting)
  in~\cite{shah2020sequential,Jain2022}.
  In particular, \cite{Jain2022} proposed a stopping rule based
  on prior-posterior-ratio (PPR) martingale confidence sequences
  \citep{waudby2020confidence} and showed that its sample complexity
  is asymptotically optimal as the target error probability goes to
  zero. We prove an analogous result for an alternate Chernoff-style
  stopping rule based on a Generalized Likelihood Ratio (GLR)
  statistic \citep{chernoff1959sequential}. More importantly, this
  result for the identityless sampling model provides a benchmark for
  the identity-based sampling model described next and will help us
  quantify the impact of identity information in community mode
  estimation.

  
\item \textit{Identity-based sampling}: \\ Under this sampling model,
  we assume that along with the community of the sampled individual,
  the learning agent is also able to \emph{identify} her, such that it
  knows whether or not this individual has been seen
  before.

  It is important to note that under identity-based sampling, the
  learning agent does not need to observe or record a unique
  identifier (say a `social security number') associated with the
  sampled individual. It may simply assign a pseudo-identity to each
  individiual when she is seen for the first time; all that is assumed
  here is the agent's ability to compare the sampled individual with
  those sampled before. For example, in the context of wildlife
  conservation, samples might correspond to animal sightings collected
  by camera traps, with image processing tools used to identify
  distinct
  sightings.\footnote{\tiny{\url{https://www.nytimes.com/1996/07/30/science/snapshots-to-improve-the-tallies-of-tigers.html}}}

  Let $\sigma_t$ be an indicator of whether the individual
  sampled at time~$t$ is new or repeated; formally,
  \begin{equation*}
    \sigma_t = 
    \begin{cases}
      0 & \text{if sample at time~$t$ was seen before,}\\ 1 & \text{if
        sample at time~$t$ was observed for the first time}.
    \end{cases}
  \end{equation*}
  Under the identity-based sampling model, the observation recorded at
  time~$t$ is $(x_t, \sigma_t)$, and the observations made till
  time~$t$ are captured by $(\bx^t, \bsigma^t),$ where $\bsigma^t =
  (\sigma_1, \sigma_2, \ldots, \sigma_t).$
\end{enumerate}
Unlike the identityless sampling model, observations under the
identity-based sampling model are not i.i.d. across time; rather, the
observation process is \emph{Markovian}. This complicates the design
and analysis of sound stopping rules for the identity-based sampling
model (see Section~\ref{sec:identity}).\footnote{The identity-based
  sampling model has been also used before in the context of
  cardinality estimation~\citep{finkelstein1998,budianu2006sensors}
  and community exploration~\citep{chen2018community}.}

As noted before, our goal is to estimate the community mode in a fixed
confidence setting under the above sampling models. Without loss of
generality, we assume that the community~1 is the largest, and that
the community sizes satisfy $d_1 > d_2 \geq \cdots d_{K-1} \geq d_K.$
Given a prescribed error threshold $\delta \in (0,1),$ we seek to
design $\delta$-PC algorithms (PC stands for \emph{Probably
  Correct}), i.e., algorithms that terminate at a stopping time, say
$\tau_\delta,$ and output a mode estimate $\hat{a}_{\tau_\delta},$
such that the probability of error is bounded by $\delta,$ i.e.,
$\pP(\tau_\delta < \infty, \hat{a}_{\tau_\delta} \neq 1 ) \leq
\delta.$


\section{Mode estimation with identityless sampling}
\label{sec:noidentity}

In this section, we study the sequential mode estimation problem under
the identityless sampling model. Specifically, we propose an algorithm
that stops sampling based on a novel generalized likelihood ratio
(GLR) statistic inspired by~\cite{karthik2020} (which uses a related
statistic in the context of detecting an odd Markov arm). This
algorithm is proved to be sound (i.e., $\delta$-PC), and also
asymptotically optimal as $\delta \da 0.$ We also propose a
`one-versus-one' variant of the same algorithm, which makes only
pairwise comparisons between communities (disregarding observations
pertaining to other communities when considering any pair of
communities) and show that this `1v1' variant outperforms the original
algorithm. The algorithms proposed in this section serve as a
benchmark for the algorithms we design under the identity-based
sampling model in Section~\ref{sec:identity}.

We begin by stating an information theoretic lower bound on the
expected stopping time of any $\delta$-PC mode estimation algorithm
under identityless sampling.
\begin{lemma}[Theorem~3 in \cite{shah2020sequential}]
  \label{lemma:lb_indentity-less}
  Under the identityless sampling model, given $\delta \in (0,1)$, the
  expected stopping time $\eE[\tau_\delta]$ of any $\delta$-PC
  algorithm satisfies
    \begin{equation*}
      \eE[\tau_\delta] \geq \frac{ \log\left( \frac{1}{2.4\delta}
      \right) }{(p_1 + p_2 )\ kl\left( \frac{p_1}{p_1 + p_2 } , 
          \frac{1}{2} \right)},
    \end{equation*}
   where $kl(x, y) := x\log\left(\frac{x}{y}\right) +
   (1-x)\log\left(\frac{1-x}{1-y}\right)$ denotes the binary relative
   entropy function.\footnote{Throughout the paper, all logarithms are
     to the base~$e.$}
\end{lemma}

Next, we describe our main algorithm, which seeks to match the above
lower bound as $\delta \da 0.$

\subsection{The \nialgo\ algorithm}

At time $t \in \nN$ (specifically, after sampling epoch~$t$), let
$N_i(t)$ denote the number of observations from community~$i,$ and let
$\hat{p}_i(t) := N_i(t) / t$ denote the empirical frequency of samples
from community~$i.$ Let $\hat{a}_t := \argmax_{i \in \K}
\hat{p}_i(t)$ and $\hat{b}_t := \argmax_{i \in \K
  \setminus\{\hat{a}_t\}} \hat{p}_i(t)$ denote the communities
associated with the largest and second largest empirical frequency
respectively. (In case either maximizer is not unique, $\hat{a}_t$ and
$\hat{b}_t$ may be chosen arbitrarily from the corresponding sets of
maximizers.)

We employ the following stopping statistic:
\begin{align}
	Z(t) & := \log\frac{B(N_1(t)+ 1, N_2(t)+1, \ldots , N_K(t)+1)}{B(1,1,\ldots,1)}  
	- \sum_{i \in \K\setminus\{\hat{a}_t,\hat{b}_t\}}N_i(t)\log(\hat{p}_i(t)) \nonumber \\
        & \qquad
        - (N_{\hat{a}_t}(t) + N_{\hat{b}_t}(t))\log\left(\frac{\hat{p}_{\hat{a}_t}(t) + \hat{p}_{\hat{b}_t}(t)}{2}\right)
        \label{eq:Zt_iless}
\end{align}
Here $B:\R_+^K \ra \R_+ $ denotes the multinomial beta
function.\footnote{For $\boldsymbol{\psi} = (\psi_1, \psi_2, \ldots,
  \psi_K) \in \R_+^K,$ $\displaystyle B(\boldsymbol{\psi}) =
  \frac{\prod_{i=1}^{K}\Gamma(\psi_i)}{\Gamma(\sum_{i=1}^{K}\psi_i)},$
  where $\Gamma(.)$ denotes the gamma function.} The proposed
algorithm, which we refer to as the \emph{No Identity Mode Estimation}
(\nialgo) algorithm, stops sampling if
\begin{equation}
  \label{eq:stopping_identityless}
  \begin{array}{l}
    Z(t) 
     > \beta(t,\delta) := \log\left(
    \frac{K-1}{\delta} \right),
  \end{array}
\end{equation}
and outputs the mode estimate~$\hat{a}_t$. This is stated formally as
Algorithm~\ref{alg:identityless}.

\begin{algorithm}
  \caption{No Identity Mode Estimation (\nialgo) algorithm}
  \label{alg:identityless}
  \begin{algorithmic}
    \STATE $t \gets 1$
    \LOOP
    \STATE Obtain $x_t$
    \STATE $\hat{a}_t \gets \arg\max_{i \in \mathcal{K}} \hat{p}_i(t) $, $\hat{b}_t \gets
    \arg\max_{i \in \mathcal{K} \setminus \{\hat{a}_t\}} \hat{p}_i(t) $
    \STATE $ Z(t) \gets \log\frac{B(N_1(t)+ 1, N_2(t)+1, \ldots , N_K(t)+1)}{B(1,1,\ldots,1)}
    - \sum_{i \in \K\setminus\{\hat{a}_t,\hat{b}_t\}}N_i(t)\log(\hat{p}_i(t))$
    \STATE $ \hspace{3cm} -(N_{\hat{a}_t}(t) + N_{\hat{b}_t}(t))\log\left(\frac{\hat{p}_{\hat{a}_t}(t) + \hat{p}_{\hat{b}_t}(t)}{2}\right)$
    \STATE $\beta(t, \delta) \gets \log\left( \frac{K-1}{\delta} \right)$
    \IF {$Z(t) >  \beta(t, \delta)$}
    \STATE { stopping time $\tau_\delta = t$ and mode $\hat{a}_{\tau_\delta} = \hat{a}_t $}
    \STATE Exit loop
    \ELSE
    \STATE{$t \gets t+1 $}
    \ENDIF
    \ENDLOOP
  \end{algorithmic}
\end{algorithm}

\subsection{Stopping rule interpretation}
From \eqref{eq:stopping_identityless}, the stopping time of the
\nialgo\ algorithm is given by
\begin{equation}
  \tau_\delta  = \inf\{t \in \nN \colon Z(t)  > \beta(t, \delta)   \}.
  \label{stopping_time} 
\end{equation}
In what follows, we interpret $Z(t),$ defined in \eqref{eq:Zt_iless}
as a certain GLR statistic.

At time $t \in \nN,$ the observations made by the learning agent thus
far are represented by $\bx^t = (x_1, x_2, \ldots, x_t).$ Recall that
$x_j$ denotes the community associated with the sampled individual at
time $j$ and $p_i = d_i / N$ denotes the probability of drawing a sample from
community~$i$ at any given time. Thus, the likelihood associated with
the observations till time~$t$ is given by
\begin{equation*}
  \mathcal{L}_{\boldsymbol{p}}(x_1, x_2, \ldots, x_t) = \prod_{i=1}^K p_i^{N_i(t)},
\end{equation*}
where $\boldsymbol{p} = (p_1,p_2, \ldots p_K).$ Let $\mathcal{S}$
denote the $K$-dimensional probability simplex, i.e. the set of all
$K$-dimensional probability mass functions, and let $\P$ denote the
Dirichlet distribution over~$\mathcal{S},$ with all concentration
parameters equal to one.\footnote{The probability density function
  associated with~$\P$ is constant and equal
  to~$\frac{1}{B(1,1,\ldots,1)} = \frac{1}{\Gamma(K)}$
  over~$\mathcal{S}.$} For communities $a,b \in \K,$ define the
statistic
\begin{equation}
  \label{eq:Zab_identityless}
  \displaystyle
  Z_{a,b}(t) := \log\frac{\eE_{\bp' \sim \P}[\mathcal{L}_{\boldsymbol{p}'}(\bx^t)] }{\max_{\bp' \in \mathcal{S} \colon p'_a
  		\leq p'_b}\mathcal{L}_{\boldsymbol{p}'}(\bx^t)}.
\end{equation}
We now show that $Z_{a,b}(t)$ admits a simpler representation in terms
of the empirical frequencies of communities.
\begin{lemma}
  \label{simple_Zabt_iless}
  At any time $t \in \nN$, if $\hat{p}_a(t) \geq \hat{p}_b(t),$ then 
  \begin{align*}
    Z_{a,b}(t) &= \log\frac{B(N_1(t)+ 1, N_2(t)+1, \ldots , N_K(t)+1)}{B(1,1,\ldots,1)}
    - \sum_{i \in \K\setminus\{a,b\}}N_i(t)\log(\hat{p}_i(t))\\ 
    & \hspace{3cm} - (N_a(t) + N_b(t))\log\left(\frac{\hat{p}_a(t) + \hat{p}_b(t)}{2}\right),\\
    Z_{b,a}(t) &= \log\frac{B(N_1(t)+ 1, N_2(t)+1, \ldots , N_K(t)+1)}{B(1,1,\ldots,1)} - \sum_{i=1}^{K}N_i(t)\log(\hat{p}_i(t))
    \leq Z_{a,b}(t).
  \end{align*}
\end{lemma}
Finally, the following lemma establishes the connection between the
statistics $\{Z_{a,b}(t)\}$ and our stopping
condition~\eqref{eq:stopping_identityless}.
\begin{lemma}
	\label{Lemma:Z}
	At any time $t \in \nN$,
        $$Z(t) = \max_{a \in \mathcal{K}} \min_{b \in \mathcal{K}\setminus\{a\}} Z_{a,b}(t) = Z_{\hat{a}_t, \hat{b}_t}(t).$$
        \ignore{
	\begin{align*}
	  Z(t) & = Z_{\hat{a}_t, \hat{b}_t}(t) = \log\frac{B(N_1(t)+ 1, N_2(t)+1, \ldots , N_K(t)+1)}{B(1,1,\ldots,1)} \\ 
	  & \hspace{3cm} - \sum_{i \in \K\setminus\{\hat{a}_t,\hat{b}_t\}}N_i(t)\log(\hat{p}_i(t)) -
	  (N_{\hat{a}_t}(t) + N_{\hat{b}_t}(t))\log\left(\frac{\hat{p}_{\hat{a}_t}(t) + \hat{p}_{\hat{b}_t}(t)}{2}\right).
	\end{align*}
        } 
	
\end{lemma}

\begin{remark}
  The statistic~$Z_{a,b}(t),$ whose relation to the stopping condition
  of the \nialgo\ algorithm is established in Lemma~\ref{Lemma:Z}, is
  a modification of the statistic defined in~\cite{karthik2020} (in
  the context of a different learning task--detecting an `odd' Markov
  arm). Specifically,~\cite{karthik2020} used a \emph{conditional}
  average likelihood (conditioned on hypothesis~$a$ dominating
  hypothesis~$b$) in place of the \emph{unconditional} average
  likelihood we use (see~\eqref{eq:Zab_identityless}). While this
  modification does not affect the soundness of the stopping criterion
  (as we establish in Theorem~\ref{thm:correctness_identityless}
  below), it enables a closed form, numerically tractable
  characterization of~$Z_{a,b}(t)$ (see
  Lemma~\ref{simple_Zabt_iless}).
\end{remark}

\begin{remark}
  It is also instructive to connect the statistic $Z_{a,b}(t)$ to the
  one used in the celebrated Track \& Stop algorithm
  (see~\cite{garivier2016optimal}) for the MAB problem. The
  corresponding statistic in \cite{garivier2016optimal} is the
  logarithm of the ratio of the maximum likelihood over instances
  where hypothesis~$a$ dominates hypothesis~$b$, to the maximum
  likelihood over instances where hypothesis~$b$ dominates
  hypothesis~$a.$ In effect, we replace the former
    maximum likelihood with an \emph{average} likelihood over a
    suitably chosen `fake' prior on the space of instances (note that
    we are \emph{not} considering a Bayesian setting where the
    instance is assumed to be drawn from this prior). This
  modification obviates the need for a certain Krichevsky-Trofimov
  (KT) bound in the soundness argument (see the proof of Theorem~10
  in~\cite{garivier2016optimal}), resulting in a tighter stopping
  threshold~$\beta.$
\end{remark}

\subsection{Correctness and asymptotic optimality of \nialgo}
We now present performance guarantees on the
\nialgo\ algorithm. First, we establish that the algorithm is sound.
\begin{theorem}
  \label{thm:correctness_identityless}
  Let $\delta \in (0,1).$ Under the identityless sampling model, the
  \nialgo\ algorithm (see Algorithm~\ref{alg:identityless})
  is $\delta$-PC.
\end{theorem}
The proof of Theorem~\ref{thm:correctness_identityless}, which uses
arguments similar to those in~\cite{garivier2016optimal,karthik2020},
is presented in~\ref{appendix:identityless}. Next, we establish that the
\nialgo\ algorithm is asymptotically optimal as $\delta \da 0.$
  \begin{theorem}
    \label{thm:optimality_identityless}
    Given $\delta \in (0,1),$ under the identityless sampling model,
    the stopping time $\tau_\delta$ (see~\eqref{stopping_time}) for the
    \nialgo\ algorithm (see Algorithm~\ref{alg:identityless})
    satisfies $\pP(\tau_\delta < \infty) =1.$
    Moreover,
    \begin{align}\label{eq:asymp_op_1_iless}
      \pP\left( \limsup_{\delta \to 0}
      \frac{\tau_\delta}{\log(1/\delta)} \leq
      (g^*(\boldsymbol{p}))^{-1}\right) = 1, 
    \end{align}
    and 
    \begin{align}\label{eq:asymp_op_2_iless}
    	 \lim\sup_{\delta \to 0} \frac{\eE_[\tau_\delta]}{\log(\frac{1}{\delta})} \leq (g^*(\boldsymbol{p}))^{-1}.
   \end{align}
    where $g^*(\boldsymbol{p}) = (p_{1} +
    p_{2})kl\left(\frac{p_{1}}{p_{1} + p_{2}}, \frac{1}{2}\right) $.
  \end{theorem}
  In light of Lemma~\ref{lemma:lb_indentity-less},
  Theorem~\ref{thm:optimality_identityless} establishes that
  \nialgo\ has an expected stopping that (multiplicatively) matches
  the information theoretic lower bound as $\delta \da 0.$ Moreover,
  the same scaling is also established in an almost sure sense. These
  results are structurally analogous to the best optimality guarantees
  that are available in the MAB setting. The proof of
  Theorem~\ref{thm:optimality_identityless}, which also uses arguments
  similar to those in~\cite{garivier2016optimal,karthik2020}, is
  presented in~\ref{appendix:identityless}.

\subsection{A one-versus-one variant of \nialgo}
\label{sec:NI-ME-1v1}

While the \nialgo\ algorithm is asymptotically optimal as $\delta \da
0,$ we now present a variant of this algorithm that performs even
better for moderate choices of~$\delta.$ Specifically, this variant
employs pairwise comparisons between communities (each such pairwise
comparison being based only on the samples generated from those
communities).

Formally, define the `pairwise' statistic $$\tilde{Z}_{a,b}(t) :=
\log(B(N_a(t) + 1,N_b(t)+1)) - (N_a(t) +
N_{b}(t))\log\left(\frac{1}{2}\right).$$ Note that
$\tilde{Z}_{a,b}(t)$ is simply what $Z_{a,b}(t)$ (as defined
in~\eqref{eq:Zab_identityless}) would be in a hypothetical instance
consisting of only the two communities~$a$ and~$b$ (with observations
from other communities being disregarded).

The proposed variant of \nialgo, which we refer to as \nialgo-1v1,
stops sampling when
\begin{align}
  \label{eq:nimeivi_stopping}
  \tilde{Z}(t) := \tilde{Z}_{\hat{a}_t,\hat{b}_t}(t) > \beta(t,\delta),
\end{align}
declaring $\hat{a}_t$ as the mode. It is easy to see that
$\tilde{Z}_{\hat{a}_t,\hat{b}_t}(t) \leq \tilde{Z}_{\hat{a}_t,b}(t)$
for all $b \neq \hat{a}_t,$ meaning the stopping
criterion~\eqref{eq:nimeivi_stopping} is equivalent to
$\tilde{Z}_{\hat{a}_t,b}(t) > \beta(t,\delta)$ for all $b \neq
\hat{a}_t.$

\begin{remark}
  \label{remark:ppr}
Remarkably, \nialgo-1v1 as defined above is identical to the PPR-1v1
algorithm proposed in~\cite{Jain2022}. The latter algorithm is
designed on the basis of very distinct machinery--confidence intervals
obtained via prior-posterior-ratios (PPRs). This suggests an
interesting connection between certain types of GLR based stopping
criteria and confidence intervals; formalizing and exploiting this
connection in the context of PAC learning is a promising avenue for
future work.
\end{remark}

We now establish performance guarantees for \nialgo-1v1. The soundness
of the algorithm follows directly from the soundness of~\nialgo.
\begin{lemma}
\label{thm:correctness_1v1_identityless}
Let $\delta \in (0,1).$ Under the identityless sampling model, the
\nialgo-1v1 algorithm is $\delta$-PC.
\end{lemma}
\begin{proof}

  Let~$t_{a,b}(k)$ denote the subsequence of observations pertaining
  only to communities~$a$ and~$b.$ The key property we exploit here is
  that under our sampling model, this subsequence can be treated as
  being generated by a 2-community instance consisting only of~$d_a$
  individuals from community~$a$ and~$d_b$ individuals from
  community~$b.$ The fact that the \nialgo\ algorithm is $\delta$-PC over such
  2-community instances implies the soundness of \nialgo-1v1 in the
  original~$K$-community instance.
  
  Let $\hat{a}$ denote the estimated mode by \nialgo-1v1.   
  \begin{align*}
    \pP(\hat{a} \neq 1) & \leq \pP( \exists b \in \mathcal{K}\setminus \{ 1\}, \exists t \in \nN, \forall a \in \mathcal{K}\setminus \{ b\}: \tilde{Z}_{b,a}(t) > \beta(t,\delta) ) \\
    & = \sum_{b \in \mathcal{K}\setminus \{ 1\}} \pP\left(\exists k \in \nN, \forall a \in \mathcal{K}\setminus \{ b\}: \tilde{Z}_{b,a}(t_{a,b}(k)) > \log\left(\frac{K-1}{\delta}\right) \right) \\
    & \overset{(a)}{=}  \sum_{b \in \mathcal{K}\setminus \{ 1\}} \pP\left(\exists k \in \nN, \forall a \in \mathcal{K}\setminus \{ b\}: Z^{(2)}_{b,a}(k) > \log\left(\frac{K-1}{\delta}\right) \right)\\
    & \overset{(b)}{\leq} \sum_{b \in \mathcal{K}\setminus \{ 1\}} \frac{\delta}{K-1} =  \delta. 
  \end{align*}
  In~$(a),$ $Z^{(2)}_{a,b}(k)$ denotes the $Z_{a,b}(\cdot)$ statistic,
  computed after~$k$ observations, in a hypothetical 2-community
  instance consisting only of~$d_a$ individuals from community~$a$
  and~$d_b$ individuals from community~$b.$ The inequality $(b)$
  follows the soundness of \nialgo\ over such two-community
  instances. (Note that the error threshold that is imposed on these
  2-community instances is $\nicefrac{\delta}{K-1}.$)
\end{proof}
Having established soundness, we now compare the stopping times under
\nialgo\ and \nialgo-1v1.
\begin{lemma}
  \label{lemma:iless_1v1_versus_nime}
  For all~$t,$ $\tilde{Z}(t) \geq Z(t).$ As a consequence, on any
  sample path, \nialgo-1v1 terminates no later than \nialgo.
\end{lemma}
\begin{proof}
  That $\tilde{Z}(t) \geq Z(t)$ follows, via a logarithmic
  transformation, from the analysis in Appendix C.3.2 of~\cite{Jain2022}.
\end{proof}
In practice, we find that \nialgo-1v1 often terminates significantly
faster than \nialgo\ for moderate choices of~$\delta;$ see
Section~\ref{sec:results}. Lemma~\ref{lemma:iless_1v1_versus_nime}
also implies that \nialgo-1v1 is asymptotically optimal, in light of
the corresponding guarantee for \nialgo.
\begin{corollary}
  The stopping time under the \nialgo-1v1 algorithm satisfies the
  asymptotic optimality guarantees of
  Theorem~\ref{thm:optimality_identityless}.
\end{corollary}


 \section{Mode estimation with identity-based sampling}
\label{sec:identity}

In this section, we turn our attention to the identity-based sampling
model, which is the main focus of this paper. Under this sampling
model, the observation process is Markovian (not i.i.d.), which makes
the design of a sound stopping criterion both methodologically and
computationally challenging. Via a novel relaxation of a GLR statistic
analogous to that used under the the identityless sampling model, we
derive a stopping statistic which provides a trade-off between
sampling efficiency and computational efficiency. The corresponding
family of algorithms is proved to be $\delta$-PC and also
asymptotically optimal as $\delta \da 0.$

We begin by establishing the following information theoretic lower
bound on the expected stopping time of any $\delta$-PC algorithm
under the identity-based sampling model.
%

\begin{lemma}\label{lemma:lb_identitybased}
  Under the identity-based sampling model, given $\delta \in (0,1)$,
  the expected stopping time $\eE[\tau_\delta]$ of any $\delta$-PC
  algorithm satisfies
  \begin{equation*}
    \eE[\tau_\delta] \geq \frac{ \log\left(\frac{1}{2.4\delta}\right) }
       {\log\left(\frac{N-d_2+d_1+1}{N}\right)}.
  \end{equation*} 
\end{lemma}
It is instructive to compare the above lower bound to the one for
identityless sampling in Lemma~\ref{lemma:lb_indentity-less}.
\begin{lemma} \label{lemma:comparison_of_LB}
  It holds that
  $$\log\left(\frac{N-d_2+d_1+1}{N}\right)  > \log\left(\frac{N-d_2+d_1}{N}\right) > (p_1 + p_2 )\ kl\left(
  \frac{p_1}{p_1 + p_2 } , \frac{1}{2} \right).$$ Consequently, for
  $\delta$-PC algorithms, the lower bound on the expected stopping
  time under the identityless sampling model (see
  Lemma~\ref{lemma:lb_indentity-less}) is greater than that under the
  identity-based sampling model (see
  Lemma~\ref{lemma:lb_identitybased}). Moreover, the ratio of former
  to the latter can be arbitrarily large, as the following inequality
  shows:
  \begin{align}
    \frac{\log\left(\frac{N-d_2+d_1+1}{N}\right)}{(p_1 + p_2 )\ kl\left(
      \frac{p_1}{p_1 + p_2 } , \frac{1}{2} \right)}   > \frac{\log\left(\frac{N-d_2+d_1}{N}\right)}{(p_1 + p_2 )\ kl\left(
      \frac{p_1}{p_1 + p_2 } , \frac{1}{2} \right)} >
    \frac{(p_1 + p_2)\log(2)}{(p_1-p_2)}
    \label{eq:lbs_ratio}
  \end{align}
\end{lemma}
The above lemma highlights that \emph{identity information can
  significantly enhance the sampling efficiency of community mode
  estimation}. From \eqref{eq:lbs_ratio}, it is easy to check that if
$p_2$ is close to $p_1$, i.e., the second largest community is only
slightly smaller than the largest community, the lower bound on the
expected stopping time in the identity-based case is much smaller
(multiplicatively) than the lower bound in the identityless case.  In
other words, for \emph{challenging} instances, identity information
can significantly speed up the task of community mode
estimation. Some intuition for the above comes from
  the analysis of the fixed budget variant of the community mode
  estimation problem (see~\cite{jain2021sequential}). In particular,
  for any fixed~$t,$ it was shown in [18] that
  \begin{align*}
    \mathbb{P}\left(\hat{a}_t \neq 1\right) &\leq c_1\ 
    \mathrm{exp}\left(-t \log\left(\frac{1}{1-(\sqrt{p_1}-\sqrt{p_2})^2}
    \right)\right), \\
    \mathbb{P}\left(\tilde{a}_t \neq 1\right) &\leq c_2\ 
    \mathrm{exp}\left(-t \log\left(\frac{1}{1-(p_1-p_2)}
    \right)\right). \\
  \end{align*}
  Here, $c_1$ and $c_2$ are instance dependent constants. Recall
  that~$\hat{a}_t$ is the empirical mode under the identityless
  setting (the community with the greatest number of samples).
  Finally, $\tilde{a}_t$ denotes the the community with the greatest
  number of \emph{distinct} samples; it may be interpreted as the
  empirical mode under the identity-based setting. (Lower bounds that
  match the above decay rates are also established
  in~\cite{jain2021sequential}.)  It is easy to see that the
  exponential decay rates above
  satisfy $$\log\left(\frac{1}{1-(\sqrt{p_1}-\sqrt{p_2})^2} \right) <
  \log\left(\frac{1}{1-(p_1-p_2)} \right).$$ This means that for
  large~$t,$ it is more likely that the empirical mode disagress with
  the true mode under the identityless setting as compared to the
  identity-based setting. This provides intuition for why identity
  information is useful in the first place. Moreover, under
  challenging instances, denoting $p_1 - p_2 = \epsilon$ for small
  $\epsilon > 0,$ the decay rate in the identity based setting is
  approximately $\log(1+\epsilon) \approx \epsilon$ whereas the decay
  rate in the identityless setting is $o(\epsilon).$ This suggests
  that the time required to make the probability of error $\leq
  \delta$ is approximately~$\frac{\log(1/\delta)}{\epsilon}$ in the
  identity-based setting, and
  approximately~$\frac{\log(1/\delta)}{o(\epsilon)}$ in the
  identityless setting; this in turn suggests that the boost provided
  by identity information is particularly strong under challenging
  instances.

While Lemma~\ref{lemma:comparison_of_LB} compares lower bounds on the
expected sample complexity, note that
Theorems~\ref{thm:optimality_identityless} and
\ref{thm:optimality_identitybased} also demonstrate their asymptotic
achievability as $\delta \to 0$. The proofs of
Lemmas~\ref{lemma:lb_identitybased} and~\ref{lemma:comparison_of_LB}
are given in \ref{appendix:identitybased}. In the remainder of this
section, we develop and analyse a family of algorithms that seek to
match the lower bound in Lemma~\ref{lemma:lb_identitybased} as $\delta
\da 0.$

\subsection{Algorithms}

Recall that under the identity-based sampling model, the observation
recorded at time $t$ is $(x_t, \sigma_t),$ where $x_t$ denotes the
community label of the sampled individual, and $\sigma_t$ is an
indicator of whether this individual is being seen for the first time. Let
$S_j(t)$ denote the number of \textit{distinct} individuals from
community~$j$ sampled till time $t$. Note that the vector $S(t) =
(S_1(t), S_2(t), \ldots, S_K(t))$ is a function of the
observations~$(\bx^t,\bsigma^t)$ recorded until time~$t.$

{\bf Definition of the stopping statistic:} The family of proposed
stopping rules is parameterized by a probability distribution~$\D$
over~$\N^K,$ which is defined as follows. Let $\theta(\cdot)$ denote a
probability mass function supported on $\N,$ i.e., it satisfies
$\theta(i) > 0 \ \forall\ i \in \N$ and $\sum_{i \in \N} \theta(i) =
1.$ Then for $\bd' \in \N^K,$ let 
\begin{equation}
\pP_{\D}\left(\bd'\right):= \prod_{i=1}^K \theta(d'_j)
\label{Eqn:Prior}
\end{equation}
 denote a prior distribution on the
community sizes. Let $\mathcal{K}^{(t)}$ denote the set of communities
that have been observed until time~$t$, i.e., $\mathcal{K}^{(t)} := \{j
\in \K \colon S_j(t) > 0\},$ and let $K^{(t)}:= |\K^{(t)}|.$ Finally, for
$\alpha \in \N,$ let $$\mathcal{S}_{\alpha}(t) := \{\bd' \in \N^K
\colon S_j(t) \leq d'_j \leq \alpha S_j(t)\ \forall\ j\}.$$

Let $\tilde{a}_t:=
\argmax_{j \in \K} S_j(t)$ denote the community that has produced
the largest number of distinct individuals and let $\tilde{b}_t:=
\argmax_{j \in \K\setminus\{\tilde{a}_t\}} S_j(t)$ denote the community with the second largest number of distinct individuals. 
The identity-based stopping statistic is defined once $t> \sum_{j\in
  \K}S_j(t) + K^{(t)}$; this means that the number of `collisions' (i.e.,
instances where the sampled individual had been seen before) exceeds
the number of communities observed. Under this condition, we define the
statistic 
\begin{align}
   Y(t) := 
  &\log\left( \sum_{\boldsymbol{d}' \in \S_{\alpha}(t)} \frac{\prod_{j=1}^{K}\prod_{l=0}^{S_j(t)-1}(d'_j-l)\pP_{\bd'\sim \D}(\boldsymbol{d}')}{(\sum_{i=1}^{K}d'_i)^t} \right)  \nonumber \\
  & \qquad \qquad \qquad \qquad \quad -\sum_{j \in \mathcal{K}^{(t)}}\int_{-1}^{S_j(t)}\log(d^{\tilde{a}_t,\tilde{b}_t}_j(\gamma_0)-v)dv + t\log\left(\sum_{i=1}^{K}d^{\tilde{a}_t,\tilde{b}_t}_i(\gamma_0)\right),
  \label{eq:Yabt}
\end{align}
where
\begin{align*}
   d^{\tilde{a}_t,\tilde{b}_t}_{\tilde{a}_t}(\gamma_0) := d^{\tilde{a}_t,\tilde{b}_t}_{\tilde{b}_t}(\gamma_0)  = \frac{ S_{\tilde{a}_t}(t) + S_{\tilde{b}_t}(t) + 2\gamma_0^2 + r(\gamma_0)}{2(1 - \gamma_0^2)}, \quad
   d^{\tilde{a}_t,\tilde{b}_t}_j(\gamma_0) := \frac{S_j(t) + \gamma_0 }{1-\gamma_0} \text{ for } j \neq \tilde{a}_t,\tilde{b}_t.
\end{align*}
Here, $\gamma_0$ is the unique zero in $(0,1)$ of
\begin{align}\label{eq:g_gamma}
  g(\gamma)  &  := g_1(\gamma) - t,
\end{align}

\begin{align*}
  \text{ where }g_1(\gamma) := \log\left(\frac{1}{\gamma}\right)\Bigg( \frac{ \sum_{j\in \mathcal{K}^{(t)}\setminus\{\tilde{a}_t,\tilde{b}_t\}}S_j(t) + \gamma(K^{(t)}-2) }{1-\gamma} 
  + \frac{ S_{\tilde{a}_t}(t) + S_{\tilde{b}_t}(t) + 2\gamma^2 + r(\gamma)}{(1 - \gamma^2)} \Bigg)
\end{align*}
and $r(\gamma) := \sqrt{ (S_{\tilde{a}_t}(t) - S_{\tilde{b}_t}(t))^2 + 4\gamma^2(1+S_{\tilde{a}_t}(t)
  +S_{\tilde{b}_t}(t)+ S_{\tilde{a}_t}(t)S_{\tilde{b}_t}(t))}.$\\

\ignore{The value of $\gamma_0$ is between $0$ and $1$. Note that,
  from lemma \ref{lemma_g_gamma}, $g_1(\gamma)$ if a decreasing
  function of $\gamma$ when $\gamma \in (0, 1)$. It follows from
  \eqref{eq:g_1_gamma_as_gamma_to_1} in Appendix
  \ref{appendix:identitybased} that for $t = \sum_{j\in \K}S_j(t) +
  K^{(t)}$, $\gamma_0 \to 1$ and as $t$ goes to infinity, $\gamma_0 \to
  0$.}

{\bf Stopping rule:} The family of algorithms we propose is
parameterized by the choice of $\alpha$ and the distribution~$\theta$
on~$\N.$ An instance of this family, referred to as the
\emph{Identity-based Community Mode Estimation} algorithm, or
\ibalgo$(\alpha,\theta),$ operates as follows.

\begin{itemize}
	\item At time~$t,$ if the stopping condition of \nialgo\ holds for
	$\beta(t,\delta/2)$, then \ibalgo$(\alpha,\theta)$ also stops
	sampling and reports the output of \nialgo\ as the mode
	estimate.\footnote{The stopping condition of \nialgo\ can be
		replaced by that of \nialgo-1v1, or indeed, by that of any
		algorithm that is $\delta$-PC in the identityless model. This
		choice does not influence the performance guarantees of
		\ibalgo$(\alpha,\theta)$.}
	\item If $t > \sum_{j\in \K}S_j(t) + K^{(t)},$ compute
	$Y(t)$.
	\ibalgo$(\alpha,\theta)$ stops sampling if
	\begin{equation}\label{eq:Yabt_stopping_rule}
		Y(t)  > 
		\beta(t,\delta/2) = \log\left( \frac{2(K-1)}{\delta} \right),
	\end{equation} %
	reporting $\tilde{a}_t$ as the estimated mode.
\end{itemize}
This procedure is formalized as Algorithm~\ref{alg:identitybased}.

\begin{algorithm}[h]
			\caption{Community mode estimation with identity-based sampling (\ibalgo)}
			\label{alg:identitybased}
			\begin{algorithmic}
					\STATE $t \gets 1$
					\LOOP
					\STATE Obtain $x_t$, $\sigma_t$
					\STATE $\hat{a}_t \gets \arg\max_{i \in \mathcal{K}} \hat{p}_i(t) $, $\hat{b}_t \gets \arg\max_{i \in \mathcal{K} \setminus \{\hat{a}_t\}} \hat{p}_i(t) $ 
				    \STATE $ Z(t) \gets \log\frac{B(N_1(t)+ 1, N_2(t)+1, \ldots , N_K(t)+1)}{B(1,1,\ldots,1)} $
				    \STATE $ \hspace{1.1cm} - \sum_{i \in \K\setminus\{\hat{a}_t,\hat{b}_t\}}N_i(t)\log(\hat{p}_i(t)) $
				    \STATE $ \hspace{1.3cm} -(N_{\hat{a}_t}(t) + N_{\hat{b}_t}(t))\log\left(\frac{\hat{p}_{\hat{a}_t}(t) + \hat{p}_{\hat{b}_t}(t)}{2}\right)$
					\STATE $\beta(t, \delta/2) \gets \log\left( \frac{2(K-1)}{\delta} \right)$
					\IF {$Z(t) > \beta(t,\delta/2)$} \STATE {stopping time $\tau_\delta = t$ and mode $\hat{a}_{\tau_\delta} = \hat{a}_t $}
					\STATE Exit loop
					\ELSE\IF {$t> \sum_{j\in \K}S_j(t) + K^{(t)}$ }
					\STATE  $\tilde{a}_t \gets \argmax_{j \in \K} S_j(t)$
					\STATE $\tilde{b}_t \gets \arg\max_{j \in \mathcal{K} \setminus \{\tilde{a}_t\}} S_j(t)$
					\STATE $ Y(t) \gets \log\left( \sum_{\boldsymbol{d}' \in \S_{\alpha}(t)} \frac{\prod_{j=1}^{K}\prod_{l=0}^{S_j(t)-1}(d'_j-l)\pP_{\bd'\sim \D}(\boldsymbol{d}')}{(\sum_{i=1}^{K}d'_i)^t} \right) $
					\STATE  $\qquad \qquad \qquad -\sum_{j \in \mathcal{K}^{(t)}}\int_{-1}^{S_j(t)}\log(d^{\tilde{a}_t,\tilde{b}_t}_j(\gamma_0)-v)dv + t\log\left(\sum_{i=1}^{K}d^{\tilde{a}_t,\tilde{b}_t}_i(\gamma_0)\right)$
					
					\IF {$Y(t) > \beta(t,\delta/2)$}
					\STATE {stopping time $\tau_\delta = t$ and mode $\hat{a}_{\tau_\delta} = \tilde{a}_t $}
					\STATE Exit loop
					\ELSE \STATE{$t \gets t+1 $}
					\ENDIF
					\ELSE \STATE{$t \gets t+1 $}
					\ENDIF
					\ENDIF
					\ENDLOOP
				\end{algorithmic}
\end{algorithm}
%
\subsection{Stopping rule interpretation}

We now provide an interpretation for the stopping
condition~\eqref{eq:Yabt_stopping_rule}. We begin by characterizing
the likelihood function under the identity-based sampling model. At
time $t+1$, given the observations recorded till time $t$, the
probability of picking a new sample from community $j\in \mathcal{K}$
is
\begin{align*}
  \pP(x_{t+1} =j, \sigma_{t+1}=1 | \bx^t , \bsigma^t) = \frac{d_j - S_j(t)}{\sum_{i=1}^K d_i} \quad\text{and}\quad \pP(x_{t+1} =j, \sigma_{t+1}=0 | \bx^t , \bsigma^t) = \frac{S_j(t)}{\sum_{i=1}^K d_i}
\end{align*}  
is the probability of picking a sample already seen before. Thus, the
likelihood of the observations till time $t$ is given by
\begin{align}
   \mathcal{L}_{\boldsymbol{d}}(\bx^t, \bsigma^t)  
  & =\underbrace{\prod_{j=1}^{K}\prod_{l=0}^{S_j(t)-1}\frac{d_j-l}{(\sum_{i=1}^{K}d_i)}}_{L_{new}} \underbrace{\prod_{\tau=1}^{t}\left(\frac{S_{x_\tau}(\tau-1)}{\sum_{i=1}^{K}d_i}\right)^{(1-\sigma_\tau)}}_{L_{rep}} \nonumber \\
  &  =\frac{\prod_{j=1}^{K}\prod_{l=0}^{S_j(t)-1}(d_j-l) \prod_{\tau=1}^{t}\left(S_{x_\tau}(\tau-1)\right)^{(1-\sigma_\tau)}}{(\sum_{i=1}^{K}d_i)^t}. \label{likelihood_with_icase}
\end{align}
where $\boldsymbol{d} = (d_1, d_2, \ldots d_K)$. The factor $L_{new}$
above arises from the new samples seen until time~$t,$ whereas the
factor $L_{rep}$ arises due to the repeated samples seen till
time~$t.$

The statistic~$Y_{a,b}(t)$ used in the proposed family of algorithms
is a relaxation of the following statistic:
\begin{align}
  \displaystyle
  & W_{a,b}(t) = \log\frac{\mathbb{E}_{\bd'\sim \D}\left[\mathcal{L}_{\boldsymbol{d}'}(\bx^t, \bsigma^t) \right]}
      {\sup_{\boldsymbol{d}'\in \N^K \colon d'_a \leq d'_b}\mathcal{L}_{\boldsymbol{d}'}(\bx^t, \bsigma^t)} .
      \label{def_Wabt} 
  %
\end{align}
Note that the numerator in \eqref{def_Wabt} (inside the logarithm) is
the average likelihood of the observation sequence with respect to the
prior distribution~$\D$ over problem instances~$\bd'$. On the other
hand, the denominator is the supremum of the likelihood over problem
instances~$\bd'$ that support the alternative hypothesis $d'_a \leq
d'_b.$ (Indeed, $W_{a,b}(t)$ is the natural generalization of
$Z_{a,b}(t),$ defined in Section~\ref{sec:noidentity}, to the
identity-based sampling model.) It can be proved that using the
statistic~$W_{a,b}(t)$ in place of $Y_{a,b}(t)$ in
Algorithm~\ref{alg:identitybased} results in a family of $\delta$-PC,
asymptotically optimal algorithms (see the proofs of
Theorems~\ref{thm:correctness_identitybased},
~\ref{thm:optimality_identitybased}).

However, using~$W_{a,b}(t)$ directly in algorithms is challenging from
a computational standpoint. In particular, computing the expectation
in the numerator in~\eqref{def_Wabt} involves an infinite sum
(over~$\N^K$). More crucially, computing the denominator
in~\eqref{def_Wabt} involves a discrete optimization over an infinite
subset of~$\N^K.$ To resolve these issues, we design a computationally
tractable lower bound $Y_{a,b}(t)$ of the statistic~$W_{a,b}(t),$
which preserves both soundness and asymptotic optimality of stopping
rule. We describe this design next, in three steps.

\noindent {\bf Step~1:} We first note that~$W_{a,b}(t)$ can be
expressed as follows.
\begin{align*}
   W_{a,b}(t)  & = \log\left(\sum_{\bd' \in \N^K} \frac{\prod_{j=1}^{K}\prod_{l=0}^{S_j(t)-1}(d'_j-l)\pP_{\boldsymbol{d}'\sim \D}(\boldsymbol{d}')}{(\sum_{i=1}^{K}d'_i)^t}\right) \nonumber \\
  &   \qquad \qquad    - \sup_{\bd' \in \N^K \colon d'_a \leq d'_b} \left\{ \sum_{j=1}^{K} \sum_{l=0}^{S_j(t)-1}\log(d'_j-l) -t\log\left(\sum_{i=1}^{K}d'_i\right)\right\} & =: T_1 - T_2.
\end{align*}
In the above representation, the factors in the likelihood that do not
depend on~$\bd'$ have been cancelled across the numerator and
denominator of~\eqref{def_Wabt}.

\noindent {\bf Step~2:} The term~$T_1$ can be lower bounded by
restricting the summation to a finite set that includes~$\bd' = S(t)$
as follows:
\begin{equation}\label{eq:Wab_T1}
  T_1 \geq \log\left(\sum_{\bd' \in \S_{\alpha}(t)} \frac{\prod_{j=1}^{K}\prod_{l=0}^{S_j(t)-1}(d'_j-l)
    \pP_{\boldsymbol{d}'\sim \D}(\boldsymbol{d}')}{(\sum_{i=1}^{K}d'_i)^t}\right)
\end{equation}
Note that the parameter~$\alpha \in \N$ controls the size of the set
over which the above summation is performed. On one extreme,
setting~$\alpha=1$ restricts the summation to a single term
corresponding to $\bd' = S(t).$ On the other extreme, as $\alpha \ra
\infty,$ we recover the term~$T_1$ exactly. Intermediate values
of~$\alpha$ provide a trade-off between computational complexity (per
sampling epoch) and sample complexity.

\noindent {\bf Step~3:} $T_2$ is upper bounded by~(i) replacing
$\sum_{l=0}^{S_j(t)-1}\log(d'_j-l)$ by the integral
relaxation~$\int_{-1}^{S_j(t)}\log(d'_j-v)dv,$ and (ii) performing the
supremum over~$\R_+^K$ rather than~$\N^K.$ Formally,
\begin{align}
  T_2 \leq 
  \sup_{\boldsymbol{d}'\in \R_+^K \colon d'_a \leq d'_b}
  \underbrace{\left\{\sum_{j \in \mathcal{K}^{(t)}}\int_{-1}^{S_j(t)}
    \log(d'_j-v)dv -t\log\left(\sum_{i=1}^{K}d'_i\right)\right\}}_{f_t(\bd')}.
  \label{eq:Wab_T2}
\end{align}
Remarkably, the above \text{continuous} optimization of the function
$f_t(\bd')$ can be reduced to a one-dimensional optimization. 
%
%
\begin{lemma}\label{lemma:sup_pf_ST_Yabt_wi}
  When communities $a, b \in \K$ satisfy $S_a(t) \geq S_b(t)$, for $t> \sum_{j\in \K}S_j(t) + K^{(t)}$, the function $f_t(\bd')$ attains
  its supremum over $\{\bd' \in \R_+^K \colon d'_a \leq d'_b\}$ at
  \begin{align*}
    d^{a,b}_a(\gamma_0) = d^{a,b}_b(\gamma_0) = \frac{ S_a(t) + S_b(t)
      + 2\gamma_0^2 + r(\gamma_0)}{2(1 - \gamma_0^2)}, \quad
    d^{a,b}_j(\gamma_0) = \frac{S_j(t) + \gamma_0 }{1-\gamma_0} \ \forall j \neq a,b,  
  \end{align*}
  where $\gamma_0$ is the unique zero in $(0,1)$ of \textcolor{blue}{$g(\gamma) = \log\left(\frac{1}{\gamma}\right)\Bigg( \frac{ \sum_{j\in \mathcal{K}^{(t)}\setminus\{a,b\}}S_j(t) + \gamma(K^{(t)}-2) }{1-\gamma} 
  + \frac{ S_{a}(t) + S_{b}(t) + 2\gamma^2 + r(\gamma)}{(1 - \gamma^2)} \Bigg) -t$.}
\end{lemma}

Lemma~\ref{lemma:sup_pf_ST_Yabt_wi} enables a tractable computation of
the upper bound in~\eqref{eq:Wab_T2}. It thus follows
from~\eqref{eq:Wab_T1},~\eqref{eq:Wab_T2}, and
Lemma~\ref{lemma:sup_pf_ST_Yabt_wi} that
\begin{align}\label{Yabt_lower_bound}
	W_{a,b}(t) \geq Y_{a,b}(t) : =   &\log\left( \sum_{\boldsymbol{d}' \in \S_{\alpha}(t)} \frac{\prod_{j=1}^{K}\prod_{l=0}^{S_j(t)-1}(d'_j-l)\pP_{\bd'\sim \D}(\boldsymbol{d}')}{(\sum_{i=1}^{K}d'_i)^t} \right)  \nonumber \\
	& \qquad \qquad \qquad \qquad \quad -\sum_{j \in \mathcal{K}^t}\int_{-1}^{S_j(t)}\log(d^{a,b}_j(\gamma_0)-v)dv + t\log\left(\sum_{i=1}^{K}d^{a,b}_i(\gamma_0)\right)
\end{align}
when $S_a(t) \geq S_b(t)$ for $t> \sum_{j\in \K}S_j(t) + K^{(t)}.$
Crucially, unlike the statistic~$W_{a,b}(t),$ its
relaxation~$Y_{a,b}(t)$ is computationally tractable. Also, note that
when $S_a(t) \geq S_b(t)$ for $t> \sum_{j\in \K}S_j(t) + K^{(t)},$ from
\eqref{Ybat},
\begin{align*}
  Y_{b,a}(t)  =   &\log\left( \sum_{\boldsymbol{d}' \in \S_{\alpha}(t)} \frac{\prod_{j=1}^{K}\prod_{l=0}^{S_j(t)-1}(d'_j-l)\pP_{\bd'\sim \D}(\boldsymbol{d}')}{(\sum_{i=1}^{K}d'_i)^t} \right)-\sup_{\boldsymbol{d}'\in \R_+^K}f_t(\bd') \leq Y_{a,b}(t). 
\end{align*}
The following lemma establishes the connection between $Y(t)$ and
$Y_{a,b}(t)$.
\begin{lemma}
  \label{Lemma:Y}
  At any time $t > \sum_{j\in \K}S_j(t) + K^{(t)},$
  \begin{align*}
    Y(t)  = \max_{a \in \K}\min_{b\in \K \setminus \{a\} } Y_{a,b}(t) = Y_{\tilde{a}_t, \tilde{b}_t}(t).
  \end{align*}
\end{lemma} 
As we show next, the proposed algorithms based on this
relaxation are $\delta$-PC and asympotically optimal as $\delta \da
0.$

\subsection{Correctness and asymptotic optimality}

We now establish that the proposed family of algorithms enjoys
guarantees analogous to those established before under identityless
sampling.
\begin{theorem}\label{thm:correctness_identitybased}
  Under the identity-based sampling model, given $\delta \in (0,1),$
  \ibalgo$(\alpha,\theta)$ is $\delta$-PC.
\end{theorem}
\begin{theorem}\label{thm:optimality_identitybased}
  Under the identity-based sampling model, given $\delta \in (0,1),$
  the stopping time $\tau_\delta$ for \ibalgo$(\alpha,\theta)$
  satisfies $\pP(\tau_\delta < \infty) =1$ and
  \begin{align}\label{eq:asymp_op_1_wi}
    \pP\left( \limsup_{\delta \to 0} \frac{\tau_\delta}{\log(1/\delta)} \leq
    \left(\log\left(\frac{N-d_2 + d_1}{N}\right)\right)^{-1}\right) = 1. 
  \end{align}
  Moreover,
  \begin{align}\label{eq:asymp_op_2_wi}
  	\limsup_{\delta \to 0}\frac{\eE[\tau_\delta]}{\log\left(\frac{1}{\delta}\right)} \leq \frac{1}{\log\left(\frac{N - d_2 + d_1}{N}\right)}.
  \end{align}
\end{theorem}
The proofs of Theorem \ref{thm:correctness_identitybased} and Theorem
\ref{thm:optimality_identitybased}, which follow via arguments
analogous to those used to prove the corresponding statements under
identityless sampling, are presented in~\ref{appendix:identitybased}.
We conclude with the following remarks.
\begin{itemize}
\item Comparing Theorem~\ref{thm:optimality_identitybased} with
  Lemma~\ref{lemma:lb_identitybased}, we see that the lower bound (on
  expected stopping time) nearly matches the (almost sure) asymptotic
  upper bound; the term $(d_1 - d_2+1)$ in the former is replaced by
  $(d_1 - d_2)$ in the latter. The same goes for the upper bound on
  expected sample complexity.  This establishes that our family of
  algorithms is asymptotic optimal in the limit~$\delta \da 0;$ the
  small discrepancy between the upper and lower bounds stems from the
  inherent discreteness of the space of problem instances.

\item The asymptotic optimality guarantee of
  Theorem~\ref{thm:optimality_identitybased} demonstrates that the
  proposed algorithms do indeed optimally exploit identity information
  for efficient community mode estimation. Interestingly, these
  algorithms are based on simply keeping track of the number of
  distinct individuals seen from each community.
  
\item The soundness and asymptotic optimality guarantees of the
  proposed family of algorithms do not depend on the choices of the
  parameters $\theta$ (a distribution suppored on~$\N$) and $\alpha$
  (a natural number). Intuitively, soundness follows because the
  statistic we employ is a lower bound of another sound statistic
  (derived using~$W_{a,b}(\cdot)$). Asymptotic optimality is a
  consequence of the (eventual) concentration of the likelihood
  function around the actual instance parameters~$d.$

\item That said, the choices of the parameters $\theta$ and $\alpha$
  do influence algorithm performance for moderate (and practical)
  values of $\delta.$ Interestingly, we find in our numerical case
  studies (see Section~\ref{sec:results}) that setting $\alpha=1$
  minimizes the computational complexity per sampling epoch, without
  significantly affecting the sampling complexity. On the other hand,
  choosing the distribution~$\theta$ to have a \emph{heavier tail}
  results in a significant reduction in the average stopping time.

  \item Finally, we note that it is also possible to design a `1v1'
    variant of \ibalgo\, using the same design principle as
    \nialgo-1v1. While it is easy to show that such a 1v1 variant is
    sound, it turns out that it does not inherit the asymptotic
    optimality properties of~\ibalgo. Mathematically, this is a
    consquence of the specific form of the information theoretic lower
    bound under the identity-based sampling model. Intuitively, the
    1v1 variant is less effective in the identity-based setting
    because, unlike in the identityless case, discriminating between
    communities~$a$ and~$b$ is aided by observations (specifically,
    collisions) from a third community~$c.$
\end{itemize}

\section{Experimental Results}
\label{sec:results}
In this section, we validate our theoretical results via numerical
experiments using both synthetic data as well as data gathered from 
real-world datasets. We compare the performance of
\nialgo\ (identityless setting), \nialgo-1v1 (which is identical to
PPR-1v1 proposed in~\citet{Jain2022}), and \ibalgo\ (identity-based
setting). Recall that in the \ibalgo\ algorithm, until $t> \sum_{j\in
  \K}S_j(t) + K^{(t)}$ (i.e., until a certain number of repeated
samples occur), \ibalgo\ relies on \nialgo\ to estimate the mode. We
also consider another variant of \ibalgo\ that piggybacks on
\nialgo-1v1 instead of \nialgo.

We set the target error probability $\delta = 0.1$ and average our
results over $100$ runs, where in each run we randomly sample from the
underlying instance and record the stopping times for all the
algorithms. We should mention that the mode is correctly identified in
almost all runs of our experiments (in other words, all algorithms are
over-conservative from the standpoint of error); we therefore focus
primarily on the average stopping time.

For \nialgo\ our prior distribution $\P$ is fixed to be the Dirichlet
distribution. But for \ibalgo, we also evaluate the effect of the
choice of prior distribution $\D$ as well as the parameter $\alpha$
which provides a trade-off between computational complexity (per
sampling epoch) and sample complexity. We consider two different prior
distributions $\D$, by letting the associated pmf $\theta$
\eqref{Eqn:Prior} be the Geometric distribution with parameters
$q=0.1$ and $q=0.9$ respectively. We consider $\alpha = 1$ and $\alpha
= 3$, where the latter requires significantly more computational
effort.
\begin{table}
\small
\begin{center}
\begin{tabular}{||c c c c ||} 
\hline
Algorithm & Parameters                                       & Average \qquad Stopping  & time  \\
		  &                                                   & $\mathcal{I}_1$  & $\mathcal{I}_2$ \\
\hline\hline
\ibalgo\                  & $\alpha =1$,  Geometric prior with $q =0.1$ & 239.4 &  413.65  \\ 
                          & $\alpha =1$,  Geometric prior with $q =0.9$ & 626.08 &  1536.1 \\
\hline
\ibalgo\ with \nialgo-1v1 & $\alpha =1$,  Geometric prior with $q =0.1$ & 209.15 &  396.49 \\ 
\hline
\nialgo\                  &                                             & 669.85 &  3968.5 \\
\nialgo-1v1/PPR-1v1               &                                             & 275.9  &  1410.3 \\
\hline    
\ibalgo\                
                         & $\delta =0.1$, $\alpha =3$,  Geometric prior with $q =0.1$ & 235.85 &  412.6500  \\  
\hline          
\end{tabular}
\caption{\label{table_1}Comparison of average stopping times for $\mathcal{I}_1 = [20,12,8,5,5]$ and $\mathcal{I}_2 = [20,16,6,4,4]$.}
\end{center}
\end{table}
We present a comparison of the performance of \nialgo,
\nialgo-1v1/PPR-1v1, \ibalgo, and \ibalgo\ with \nialgo-1v1 in
Table~\ref{table_1} in terms of the average stopping time. We consider
two instances $\mathcal{I}_1 = [20,12,8,5,5]$ and $\mathcal{I}_2 =
[20,16,6,4,4]$, both with $50$ individuals and $5$ communities
each. Note that $\mathcal{I}_2$ is a harder instance than
$\mathcal{I}_1$ since the two largest community sizes are
closer. Comparing the two identityless sampling algorithms, we find
that \nialgo\ empirically performs worse than \nialgo-1v1/PPR-1v1, as
anticipated. The average stopping time for both instances
$\mathcal{I}_1$ and $\mathcal{I}_2$ is lowest for \ibalgo\ with
\nialgo-1v1 which uses identity-based sampling, with $\alpha =1$ and the 
Geometric prior with $q =0.1$. For the easier instance $\mathcal{I}_1$
the average stopping time for \ibalgo, for the heavier-tailed prior
(Geometric prior with $q =0.1$) is comparable to the average stopping
time of PPR-1v1. Thus for the easier instance, the identity
information does not seem to provide a significant boost in mode
estimation. However, for the harder instance $\mathcal{I}_2$, we can
see from Table \ref{table_1} that \ibalgo\ performs much better than
\nialgo-1v1/PPR-1v1 with the heavier-tailed prior. This shows that
identity information plays a vital role in reducing the stopping time
as the instance gets harder.

We also note that compared to a sharply decreasing prior distribution
like the Geometric prior with $q=0.9$, a prior distribution with heavier
tail like the Geometric prior with $q=0.1$, significantly reduces the
average stopping time of \ibalgo. Also, we observe that increasing the
value of $\alpha$ from $1$ to $3$ does not seem to have a significant
impact on the average stopping time. We could not evaluate even larger
values of $\alpha$ due to the sharp increase in computational
complexity.

We also compare the performance of the various schemes in the scenario where the target error probability $\delta$ as well as the ratios of the community sizes is fixed, but the 
overall population size $N$ is scaled by a factor of $\omega$. We choose $\alpha =1$ 
and a Geometric prior with 
$q =0.1$ for \ibalgo\ and \ibalgo\ with \nialgo-1v1. 
The results for various values of $\omega$ are presented in Table \ref{table_2}. We find that for small to moderate values of $N$, identity-based mode 
estimation algorithms \ibalgo\ and \ibalgo\ with \nialgo-1v1 
perform better than the identityless sampling algorithms \nialgo-1v1 and \nialgo. 
But as $N$ becomes large, the performance of both \ibalgo\ with \nialgo-1v1 and 
\nialgo-1v1 become comparable. Interestingly however, we should note that our analytical results (see Lemma \ref{lemma:comparison_of_LB}) imply that the multiplicative gap
between the \emph{lower bounds} in the identityless and identity-based
settings holds even when the population size $N$ is large. We are not
entirely sure whether the contrast between the lower bounds and
empirical performance is due to (i) an inefficiency of our algorithms,
or (ii) due to a looseness in our lower bounds. Unfortunately, our
methodological machinery is designed for the scaling where the
instance (and therefore~$N$) is fixed and $\delta \downarrow 0,$ and
not for the scaling where~$\delta$ is fixed and $N$ is scaled up. 
Finally, it is also important to note that $N$ is assumed to be
unknown to the learner. So in principle, even the absence of
collisions (which is likely when $N$ is large) provides additional
information to the learner in the identity-based sampling model; thus it
might be possible to design an algorithm that exploits this
information better (our algorithm admittedly does not).
\begin{figure}[t]
	\begin{center}
		\includegraphics{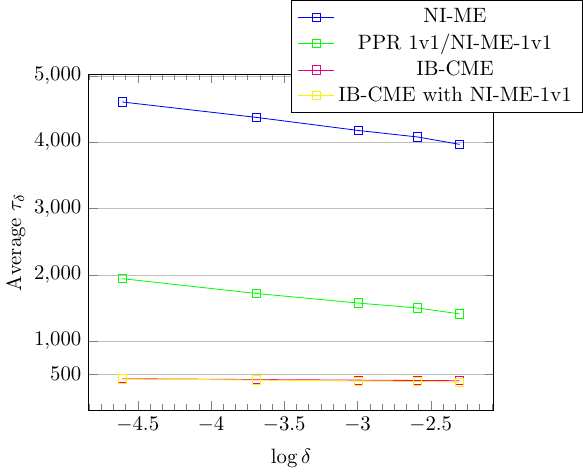}
		\caption{\label{plot}Average stopping time $\tau_\delta$ v/s $\log(\delta)$ for $\bd=[20,16,6,4,4]$.}
	\end{center}
\end{figure}

Finally, we present the effect of varying the target error probability $\delta$ on the average 
stopping time of the various schemes in Figure \ref{plot}. We consider 
the instance $\mathcal{I}_2$, and choose $\alpha =1$ and a Geometric prior with 
$q =0.1$ for \ibalgo\ and \ibalgo\ with \nialgo-1v1. As expected, the sample complexity varies linearly with $\log(\delta)$. 
 
	\begin{table}
		\small
		\begin{center}
			\begin{tabular}{|| c c c c c||} 
				\hline
				Algorithm                  & \ibalgo\         &\ibalgo\   & \nialgo-1v1/PPR-1v1 & \nialgo\\
				                           & with \nialgo-1v1         &          &         &              \\
				\hline
				Average Stopping time for  &                  &          &         &              \\
				$\omega = 1,  N=50$        & 710.535           & 711     & 6353.5  &  19705       \\
				$\omega = 5,  N=250$       & 1911.7           & 2008.1   & 6480.3  &  19707       \\
				$\omega = 10, N=500$       & 3170.1           & 3538.9   & 6260.7  &  19518       \\
				$\omega = 15, N=750$       & 4475.4           & 5076.1   & 6550.8  &  19727       \\
				$\omega = 20, N=1000$      & 5262.5           & 6590.5   & 6220.5  &  20559       \\
				$\omega = 25, N=1250$      & 6019.8           & 8077.9   & 6582.7  &  19873       \\
				$\omega = 30, N=1500$      & 6344.6           & 9593.5   & 6220.5  &  20559       \\
				$\omega = 35, N=1750$      & 6885.6           & 10995    & 6550.8  &  19727       \\
				$\omega = 40, N=2000$      & 7102.3           & 12367    & 6550.8  &  19727       \\
				\hline      
			\end{tabular}
			\captionof{table}{\label{table_2}Comparison of average stopping times for $\mathcal{I}_3 = \omega[20,18,6,3,3].$}
		\end{center}
	\end{table}

\begin{table}
\small
	\begin{center}
		\begin{tabular}{|| c c c c c||} 
			\hline
			Algorithm             & \ibalgo\         &\ibalgo\  & \nialgo-1v1/PPR-1v1 & \nialgo  \\
			                      & with \nialgo-1v1 &          &                     &          \\
			                      \hline
			 &                  & Average Stopping time         &                     &          \\
			$\text{Dataset}-1$    & 1411.7           & 1492.7   & 2061.3              & 13284    \\
			$\text{Dataset}-2$    & 3081.5           & 3717.1   & 4430.9              & 28166    \\
			$\text{Dataset}-3$    & 554.52           & 3049.2   & 505.98              & 3344.2   \\
			\hline      
		\end{tabular}
		\captionof{table}{\label{table_3}Comparison of average stopping times for real-world datasets.}
	\end{center}
\end{table}

We also evaluate the various schemes on instances derived from three 
 	real-world datasets of different 
 	sizes. For each dataset, we choose $\alpha =1$ and a Geometric prior with $q =0.1$ for 
 	\ibalgo\ and \ibalgo\ with \nialgo-1v1. The first dataset lists food orders placed by customers on an online portal across $179$ restaurants 
 belonging to $14$ different cuisines~\cite{NYC}. We refer to this dataset as 
 	$\text{Dataset}-1$. Mapping this data to our framework, the restaurants correspond 
 to the individual entities and the corresponding cuisines represents the communities. 
 The underlying instance created this way is given by $[41, 16, 3, 14, 31, 29, 5, 5, 11, 7, 2,3,9,3]$ 
 and the goal is to estimate the  cuisine with largest number of restaurants by using random sampling. 

    Next we consider a dataset, referred to as 
	$\text{Dataset}-2$ henceforth, which contains a list of $500$ top cities in $12$ different 
	regions of the world with respect to their innovation index \cite{500_cities}. We are interested in 
	finding the region with most number of innovative cities. Here, cities correspond to 
	the individual entities whereas the region they belong to represent the communities. 
	The instance created from this dataset is given by $[19, 39, 14, 44, 18, 139, 13, 12, 39, 25, 20, 118]$. 
	The aim is to determine the region with maximum number of innovative cities by using random 
	sampling. Finally, we consider $\text{Dataset}-3$, which consists of $1000$ top rated IMDb movies under $14$ diffent genre \cite{IMDb}. Here the movies are individual entities and their 
	genres are the communities. The instance 
	created from the dataset is given by $[172, 72, 82, 88, 155, 107, 289, 2, 2, 3, 11, 12, 1, 4]$. 
	The goal here is to determine the most popular genre i.e. the genre with maximum number of 
	movies listed.
	
	Note that the three datasets have increasing population sizes and while the first two datasets have similar gaps between the top two community sizes, the third dataset has a much larger gap indicating an easier problem instance. The performance results for the various mode estimation algorithms when used over  the 
	above three datasets are presented in Table~\ref{table_3}. For $\text{Dataset}-1$ and $\text{Dataset}-2$ we can see that \ibalgo\ with 
	\nialgo-1v1 algorithm has the lowest average stopping time. The identity-based algorithms perform 
	better than the identityless sampling algorithms, which 
	illustrates the significant reduction in sample complexity that identity information can provide. 
    The \nialgo-1v1 algorithm performs better than the \nialgo\ algorithm as expected. However, for $\text{Dataset}-3$ which  represents the easier problem instance, 
    the performance of algorithm \ibalgo\ with 
    \nialgo-1v1 is comparable to \nialgo-1v1 algorithm with \nialgo-1v1 performing slightly better. 
    The performance of the \ibalgo\ algorithm is only slightly better than the \nialgo\ algorithm, which is likely due to the fact that for a larger population size $N$ the advantage provided by identity information is reduced due 
    to lack of collisions or repeated samples.

\section{Concluding remarks}
\label{sec:conclusion}

In this paper, we considered the learning task of online community
mode estimation in a fixed confidence setting, focusing on the role of
identity information. Our analysis relies heavily on the existing
GLR-based machinery for establishing correctness and asymptotic
optimality of PAC algorithms (most notably, the methodology developed
in~\cite{garivier2016optimal,karthik2020}). Our key novelty is: \\ (i)
From a modeling standpoint, we highlight the value of \emph{identity}
information in mode estimation.\\ (ii) We modify the GLR statistic
proposed in \cite{karthik2020} to design a tractable algorithm for the
identityless sampling model. A superior `1v1' variant of this
algorithm is also designed, which turns out to be identical to the
state of the art algorithm for (identityless) mode estimation--PPR-1v1
(see \cite{Jain2022}), which was designed using a different design
recipe.\\ (ii) We design a sound stopping condition under a
\emph{Markovian} observation sequence (arising under identity-based
sampling), and further design a computationally tractable relaxation
of the stopping statistic that preserves both soundness and asymptotic
optimality (Lemma~\ref{lemma:sup_pf_ST_Yabt_wi} is crucial to this
relaxation); indeed, that the discrete likelihood optimization over
$\mathbb{N}^K$ in~\eqref{def_Wabt} can be tractably relaxed is far from trivial.

\bibliographystyle{plainnat}
\bibliography{example_paper}
\clearpage
\appendix

\section{Identityless sampling}
\label{appendix:identityless}
\begin{proof}[Proof of Lemma~\ref{simple_Zabt_iless}]
The statistic
\begin{equation*}
	\displaystyle
	Z_{a,b}(t) := \log\frac{ \eE_{\bp' \sim \P}[\mathcal{L}_{\boldsymbol{p}'}(\bx^t)] }{\max_{\bp' \in \mathcal{S}: p'_a\leq p'_b}\mathcal{L}_{\boldsymbol{p}'}(\bx^t)}.
\end{equation*}
 For the Dirichlet prior on $\boldsymbol{p}'\in \mathcal{S}$,  
 \begin{align}  
 	 \eE_{\bp' \sim \P}[\mathcal{L}_{\boldsymbol{p}'}(\bx^t)] 
 	& = \int_{\boldsymbol{p}' \in \mathcal{S}} \prod_{i=1}^{K}p_i^{N_i(t)}.\frac{1}{B(1,1,\ldots,1)}dp'_1dp'_2\ldots dp'_K \nonumber \\
 	& \overset{(b)}{=} \frac{B(N_1(t)+ 1, N_2(t)+1, \ldots , N_K(t)+1)}{B(1,1,\ldots,1)} \int_{\boldsymbol{p}' \in \mathcal{S}} \frac{\prod_{i=1}^{K}p_i^{N_i(t)}dp'_1dp'_2\ldots dp'_K}{B(N_1(t)+ 1, N_2(t)+1, \ldots , N_K(t)+1)}. \nonumber \\
 	& \overset{(c)}{=} \frac{B(N_1(t)+ 1, N_2(t)+1, \ldots , N_K(t)+1)}{B(1,1,\ldots,1)}. \label{exp_Num_identityless}
 \end{align} 
 In $(b)$, $B(N_1(t)+ 1, N_2(t)+1, \ldots , N_K(t)+1)$ is the normalization factor for a Dirichlet 
 distribution with paramaters $\psi_i = N_i(t) + 1$ for all $i \in \K$. The equality $(c)$ follows 
 as the integrad in $(b)$ is the pdf of a Dirichlet distribution with parameters $\psi_i = N_i(t) + 1$ 
 for all $i \in \K$. So, being the integral of a pdf over the entire domain, it integrates to 1. Thus, 
 \begin{equation*}
 	\displaystyle
 	Z_{a,b}(t) = \log\frac{B(N_1(t)+ 1, N_2(t)+1, \ldots , N_K(t)+1)}{B(1,1,\ldots,1)} - \max_{\bp' \in \mathcal{S} \colon p'_a
 		\leq p'_b}\log(\mathcal{L}_{\boldsymbol{p}'}(\bx^t)). 
 \end{equation*}

  Consider the maximization: $\max_{\bp' \in \mathcal{S} \colon p'_a
    \leq p'_b}\log(\mathcal{L}_{\boldsymbol{p}'}(\bx^t))$ when $\hat{p}_a(t) \geq \hat{p}_b(t)$. This is a
  convex optimization. Denoting the Lagrange multiplier associated
  with the constraint $p'_a \leq p'_b$ by $\lambda,$ and the Lagrange
  multiplier associated with the simplex constraint by $\nu,$ the KKT
  conditions imply that the primal optimal point $\boldsymbol{p}^*$
  and dual optimal points $\lambda^*$, $\nu^*$ satisfy:

  \begin{align*}
    p^*_a & \leq p^*_b\\
    \lambda^* & \geq 0\\
    \lambda^* (p^*_a - p^*_b) & = 0\\
    \sum_{i=1}^{K}p^*_i - 1 & = 0 \\
    -\frac{N_i(t)}{p^*_i} + \nu^* & = 0  \hspace{1cm} \forall i \in \mathcal{K}\setminus \{ a,b\} \\
    -\frac{N_a(t)}{p^*_a} + \lambda^* + \nu^* &= 0 \\ 
    -\frac{N_b(t)}{p^*_b}  - \lambda^* +  \nu^* &= 0.
  \end{align*}
  From first three conditions we can see that either $\lambda^* = 0$
  (then the problem reduces to unconstrained maximization of the likelihood function) or $p^*_a - p^*_b = 0$. If $\lambda^*=0$
  then $p^*_i = \hat{p}_i(t)$. This would violate the constraint
  $p'_a\leq p'_b$ when $\hat{p}_a(t) \geq
  \hat{p}_b(t)$. Taking $p^*_a = p^*_b$ and solving the gradient
  equations we get $ \nu^* = t$, $\lambda^* =
  t\frac{N_b(t)-N_a(t)}{N_b(t) + N_a(t)}$ and $ p^*_a = p^*_b =
  \frac{N_a(t) + N_b(t)}{2t} = \frac{\hat{p}_a(t) +
    \hat{p}_b(t)}{2}$. Therefore,
	 	\begin{equation}
	 		\max_{\bp' \in \mathcal{S} \colon p'_a
	 			\leq p'_b}\log(\mathcal{L}_{\boldsymbol{p}'}(\bx^t)) = \sum_{i \in \mathcal{K}\setminus \{ a,b\} } N_i(t) \log(\hat{p}_i(t)) + (N_a(t) + N_b(t))\log\left(\frac{\hat{p}_a(t) + \hat{p}_b(t)}{2}\right). \label{part2_zab}
	 	\end{equation}
	 	Thus $Z_{a,b}(t)$ simplifies as 
	 	  \begin{align*}
	 		Z_{a,b}(t) &= \log\frac{B(N_1(t)+ 1, N_2(t)+1, \ldots , N_K(t)+1)}{B(1,1,\ldots,1)} \nonumber \\ 
	 		& \hspace{3cm} - \sum_{i \in \K\setminus\{a,b\}}N_i(t)\log(\hat{p}_i(t)) -
	 		(N_a(t) + N_b(t))\log\left(\frac{\hat{p}_a(t) + \hat{p}_b(t)}{2}\right).
	 	\end{align*}
 	    From the definition of $Z_{a,b}(t)$ and \eqref{exp_Num_identityless}, when $\hat{p}_a(t) \geq \hat{p}_b(t)$,
 	    \begin{align}
 	    	Z_{b,a}(t) & = \log\frac{B(N_1(t)+ 1, N_2(t)+1, \ldots , N_K(t)+1)}{B(1,1,\ldots,1)} - \max_{\bp' \in \mathcal{S} \colon p'_b
 	    		\leq p'_a}\log(\mathcal{L}_{\boldsymbol{p}'}(\bx^t)) \nonumber \\
 	    	& \overset{(a)}{=} \log\frac{B(N_1(t)+ 1, N_2(t)+1, \ldots , N_K(t)+1)}{B(1,1,\ldots,1)} - \sum_{i=1}^{K}N_i(t)\log(\hat{p}_i(t)). 
 	    \end{align}
 	    In the above equation, $(a)$ follows from the fact that when $\hat{p}_a(t) \geq \hat{p}_b(t)$, $\max_{\bp' \in \mathcal{S} \colon p'_b
 	    \leq p'_a}\log(\mathcal{L}_{\boldsymbol{p}'}(\bx^t))$ reduces to unconstrained maximization of the Log-Likelihood, $\log(\mathcal{L}_{\boldsymbol{p}'}(\bx^t))$. 
        This is achieved at $p'_i = \hat{p}_i(t)$ for all $i \in \K$. 
\end{proof}


\begin{proof}[Proof of Lemma~\ref{Lemma:Z}]
When $\hat{p}_a(t) \geq \hat{p}_b(t)$, from Lemma~\ref{simple_Zabt_iless},
\begin{align*}
	\displaystyle
	Z_{a,b}(t) & = \log\frac{B(N_1(t)+ 1, N_2(t)+1, \ldots , N_K(t)+1)}{B(1,1,\ldots,1)} \nonumber \\ 
	& \hspace{3cm} - \sum_{i \in \K\setminus\{a,b\}}N_i(t)\log(\hat{p}_i(t)) -
	(N_a(t) + N_b(t))\log\left(\frac{\hat{p}_a(t) + \hat{p}_b(t)}{2}\right).
\end{align*}
and 
\begin{align}
	Z_{b,a}(t) 
	= \log\frac{B(N_1(t)+ 1, N_2(t)+1, \ldots , N_K(t)+1)}{B(1,1,\ldots,1)} - \sum_{i=1}^{K}N_i(t)\log(\hat{p}_i(t)). \label{Zbat}
\end{align}
Comparing $Z_{a,b}(t)$ with $Z_{b,a}(t)$ we get
\begin{equation}
	Z_{a,b}(t) >  Z_{b,a}(t) \text{ when }  \hat{p}_a(t) \geq \hat{p}_b(t) \label{Zab_vs_Zba}.
\end{equation}
Also, note that for any pair $a,b$, when $\hat{p}_a(t) \geq \hat{p}_b(t)$, $Z_{b,a}(t)$ has the same value. Now, we compute
\begin{align*}
	Z(t) & = \max_{a \in \K}\min_{b \in \K\setminus\{a\}}Z_{a,b}(t) \\
	& = \max\left\{\min_{b \neq \hat{a}_t}Z_{\hat{a}_t,b}(t), \max_{x \neq \hat{a}_t}\min_{b \neq x}Z_{x,b}(t)\right\} \\
	& \overset{(b)}{=} \max\left\{ Z_{\hat{a}_t, \hat{b}_t}(t), \max_{x \neq \hat{a}_t}Z_{x,\hat{a}_t}(t) \right\} \\
	& \overset{(c)}{=} \max\left\{ Z_{\hat{a}_t, \hat{b}_t}(t), Z_{\hat{b}_t,\hat{a}_t}(t) \right\} \\
	& \overset{(e)}{=} Z_{\hat{a}_t, \hat{b}_t}(t).
\end{align*}
Equality $(b)$ follows because 
$$\max_{x \neq \hat{a}_t}\min_{b \neq x}Z_{x,b}(t) = \max_{x \neq \hat{a}_t}\min_{b \neq x}\left\{Z_{x,\hat{a}_t}(t), \min_{b \neq \hat{a}_t}Z_{x,b}(t)\right\}$$ and 
$Z_{x,\hat{a}_t}(t)$ has the lowest value that $Z_{a,b}(t)$ can take (see \eqref{Zbat}). Equality $(c)$ follows as $Z_{x,\hat{a}_t}(t)$ has the same value given in \eqref{Zbat} for all $x$. The last equality $(e)$ follows from \eqref{Zab_vs_Zba}.

\end{proof}

	 \begin{proof}[Proof of Theorem~\ref{thm:correctness_identityless}]
	 	     	Let $a^*$ denote the estimated mode. For $b \in \K$, let  
	 	$$\mathit{E}_b^t = \{\omega: \tau_\delta =t, a^* = b\}$$ 
	 	be the set of sample paths for which the stopping time of \nialgo\ is $t$ and the mode estimated is $b$. Note that 
	 	$$\mathit{E}_b^{t_1}\cap\mathit{E}_b^{t_2} = \phi \text{ for all } t_1 \neq t_2.$$
	 	The probability of error in estimating the mode
	 	\begin{align*}
	 		\pP(a^* \neq 1) & = \sum_{b \in \mathcal{K}\setminus\{1\}} \sum_{t=0}^{\infty} \sum_{\omega \in \mathit{E}_b^t} \mathcal{L}_{\boldsymbol{p}}(\omega) \\
	 		& \leq  \sum_{b \in \mathcal{K}\setminus\{1\}} \sum_{t=0}^{\infty} \sum_{\omega \in \mathit{E}_b^t} \max_{p'_1 \geq p'_b}\mathcal{L}_{\boldsymbol{p'}}(\omega) \\
	 		& = \sum_{b \in \mathcal{K}\setminus\{1\}} \sum_{t=0}^{\infty} \sum_{\omega \in \mathit{E}_b^t}   \frac{\max_{p'_1 \geq p'_b}\mathcal{L}_{\boldsymbol{p'}}(\omega)}{\eE_{\bp' \sim \P}[\mathcal{L}_{\boldsymbol{p}'}(\omega)]} \eE_{\bp' \sim \P}[\mathcal{L}_{\boldsymbol{p}'}(\omega)] \\
	 		& \overset{(a_1)}{=} \sum_{b \in \mathcal{K}\setminus\{1\}} \sum_{t=0}^{\infty} \sum_{\omega \in \mathit{E}_b^t} \exp(-Z_{b,1}(t)) \eE_{\bp' \sim \P}[\mathcal{L}_{\boldsymbol{p}'}(\omega)] \\
	 		& \overset{(a_2)}{\leq} \sum_{b \in \mathcal{K}\setminus\{1\}} \exp{-\beta(t,\delta)} \underbrace{\sum_{t=0}^{\infty} \sum_{\omega \in \mathit{E}_b^t}  \eE_{\bp' \sim \P}[\mathcal{L}_{\boldsymbol{p}'}(\omega)] }_{(p*)} \\
	 		& \overset{(a_3)}{\leq} \delta.
	 	\end{align*}
	 	Equality $(a_1)$ follows from the definition of $Z_{a,b}(t)$. Inequality $(a_2)$ follows as $Z_{b,1}(t)$ will have crossed the threshold $\beta(t,\delta)$ at the stopping time $t$ when mode is $b$. Inequality $(a_3)$ follows by taking $\beta(t,\delta) = \log\left(\frac{K-1}{\delta}\right)$ as $(p*)$ is $\pP( a^* = b)$ under an alternate probability distribution.
	 \end{proof}


\begin{proof}[Proof of Theorem~\ref{thm:optimality_identityless}]
\textit{proof of~\eqref{eq:asymp_op_1_iless}}:\\
Let $\mathcal{E}$ be the event
$$\mathcal{E} = \{ \forall i \in \mathcal{K}, \hat{p}_i(t) \xrightarrow[t \to \infty]{} p_i \}.$$
By the law of Large numbers the event $\mathcal{E}$ is of probability 1. 
For the Dirichlet prior, 
\begin{align}
	Z(t) & = \log\frac{B(N_1(t)+ 1, N_2(t)+1, \ldots , N_K(t)+1)}{B(1,1,\ldots,1)} \\ 
	& \hspace{3cm} - \sum_{i \in \K\setminus\{\hat{a}_t,\hat{b}_t\}}N_i(t)\log(\hat{p}_i(t)) -
	(N_{\hat{a}_t}(t) + N_{\hat{b}_t}(t))\log\left(\frac{\hat{p}_{\hat{a}_t}(t) + \hat{p}_{\hat{b}_t}(t)}{2}\right) \nonumber \\
	& = \log\frac{\prod_{j=1}^{K}\Gamma(N_j(t)+1)\Gamma(K)}{\Gamma(t+K)} \nonumber \\ 
	& \hspace{3cm} - \sum_{i \in \K\setminus\{\hat{a}_t,\hat{b}_t\}}N_i(t)\log(\hat{p}_i(t)) -
	(N_{\hat{a}_t}(t) + N_{\hat{b}_t}(t))\log\left(\frac{\hat{p}_{\hat{a}_t}(t) + \hat{p}_{\hat{b}_t}(t)}{2}\right) \nonumber \\
	& = \log((K-1)!) + \sum_{j=1}^{K}\log(N_j(t)!) - \log((t+K-1)!) \nonumber \\
	& \hspace{3cm}- \sum_{i \in \K\setminus\{\hat{a}_t,\hat{b}_t\}}N_i(t)\log(\hat{p}_i(t)) -
	(N_{\hat{a}_t}(t) + N_{\hat{b}_t}(t))\log\left(\frac{\hat{p}_{\hat{a}_t}(t) + \hat{p}_{\hat{b}_t}(t)}{2}\right). \label{NEW_Zabt_RHS}
\end{align}
By Stirling's approximation, 
\begin{align*}
	N_j(t)! & \sim \sqrt{2\pi N_j(t)}\left(\frac{N_j(t)}{e}\right)^{N_j(t)}\\
	(t+K-1)! & \sim \sqrt{2\pi (t+K-1)}\left(\frac{t+K-1}{e}\right)^{t+K-1}
\end{align*}
as $t \to \infty.$ Using Stirling's approximation, the RHS of \eqref{NEW_Zabt_RHS} becomes 
\begin{align}
	& \log((K-1)!) + \sum_{j=1}^{K}\log(N_j(t)!) - \log((t+K-1)!) \nonumber \\
	& \hspace{1cm}- \sum_{i \in \K\setminus\{\hat{a}_t,\hat{b}_t\}}N_i(t)\log(\hat{p}_i(t)) -
	(N_{\hat{a}_t}(t) + N_{\hat{b}_t}(t))\log\left(\frac{\hat{p}_{\hat{a}_t}(t) + \hat{p}_{\hat{b}_t}(t)}{2}\right) \nonumber \\
	& = \log((K-1)!) + \sum_{j=1}^{K}\log\left(\sqrt{2\pi N_j(t)}\right) + \sum_{j=1}^{K}N_j(t)\log(N_j(t)) - t \nonumber \\ 
	& \hspace{1cm} -\log\left(\sqrt{2\pi (t + K-1)}\right) - (t + K -1)\log(t + K -1) + t + (K-1) \nonumber \\
	& \hspace{2cm} - \sum_{i \in \K\setminus\{\hat{a}_t,\hat{b}_t\}}N_i(t)\log(\hat{p}_i(t)) -
	(N_{\hat{a}_t}(t) + N_{\hat{b}_t}(t))\log\left(\frac{\hat{p}_{\hat{a}_t}(t) + \hat{p}_{\hat{b}_t}(t)}{2}\right) \nonumber \\
	& = \sum_{j=1}^{K}N_j(t)\log\left(\frac{N_j(t)}{t+K-1}\right)  - \sum_{i \in \K\setminus\{\hat{a}_t,\hat{b}_t\}}N_i(t)\log(\hat{p}_i(t)) -
	(N_{\hat{a}_t}(t) + N_{\hat{b}_t}(t))\log\left(\frac{\hat{p}_{\hat{a}_t}(t) + \hat{p}_{\hat{b}_t}(t)}{2}\right) \nonumber \\
	& \hspace{1cm} +  \log((K-1)!) + \sum_{j=1}^{K}\log\left(\sqrt{2\pi N_j(t)}\right)
	  -\log\left(\sqrt{2\pi (t + K-1)}\right) - ( K -1)\log(t + K -1) + (K-1) \label{Zabt_stirling}. 
\end{align}
On $\mathcal{E}$, as $t \to \infty$, 
\begin{align}
	\frac{N_j(t)}{t+K-1} \to \hat{p}_j(t)  \text{ and } \hat{p}_j(t) \to p_j. 
\end{align}
So, as $t \to \infty$, from \eqref{NEW_Zabt_RHS} and \eqref{Zabt_stirling}
\begin{align*}
	\frac{Z(t)}{t} \to \sum_{j=1}^{K}p_j\log\left(p_j\right)  - \sum_{i \in \K\setminus\{1,2\}}p_j\log(p_i) -(p_1 + p_2)\log\left(\frac{p_1 + p_2}{2}\right). 
\end{align*}

Simplifying the limit above we get,
\begin{align}
	& \sum_{j=1}^{K}p_j\log\left(p_i\right)  - \sum_{i \in \K\setminus\{1,2\}}p_j\log(p_i) -(p_1 + p_2)\log\left(\frac{p_1 + p_2}{2}\right) \nonumber \\
	& = p_1\log(p_1) + p_2\log(p_2) -(p_1 + p_2)\log\left(\frac{p_1 + p_2}{2}\right) \nonumber \\
	& =  (p_1 + p_2)\frac{p_1}{(p_1 + p_2)}\log\left(\frac{\frac{p_1}{(p_1 + p_2)}}{2}\right)  +  (p_1 + p_2)\frac{p_2}{(p_1 + p_2)}\log\left(\frac{\frac{p_2}{(p_1 + p_2)}}{2}\right) \nonumber \\
	& = (p_1 + p_2)kl\left(\frac{p_1}{p_1 + p_2}, \frac{1}{2}\right). \nonumber 
\end{align}
where $kl(x, y) = x\log(\frac{x}{y}) + (1-x)\log(\frac{1-x}{1-y})$ is the KL divergence of Bernoulli distribution.
So, on $\mathcal{E}$ as $t \to \infty$, 
\begin{align*}
	\frac{Z(t)}{t} \to (p_1 + p_2) kl\left(\frac{p_1}{p_1 + p_2}, \frac{1}{2}\right).
\end{align*}
So, there exists a time $t_0$ such that for $t>t_0$, and some $\epsilon >0$ ,
\begin{equation}
	Z(t) > \frac{t}{1+\epsilon}(p_1 + p_2) kl\left(\frac{p_1}{p_1 + p_2}, \frac{1}{2}\right)
\end{equation}

Let $g^*(\boldsymbol{p}) = (p_{1} + p_{2})d\left(\frac{p_{1}}{p_{1} + p_{2}}, \frac{1}{2}\right).$
Consequently, 
\begin{align*}
	\tau_\delta & = \inf\{ t \in \nN: Z(t) \geq \beta(t, \delta) \} \nonumber \\
	& \leq t_0 \vee \inf\left\{ t \in \nN: t(1 + \epsilon)^{-1}g^*(\boldsymbol{p}) \geq \log\left( \frac{(K-1)}{\delta} \right) \right\} \\
	& \leq t_0 \vee \log\left(\frac{K-1}{\delta}\right)(1+\epsilon)(g^*(\boldsymbol{p}))^{-1}
\end{align*}
Thus $\tau_\delta$ is finite on $\mathcal{E}$ for every $\delta \in (0,1)$ and by letting $\epsilon$ go to zero we get
\begin{align*}
	\lim\sup_{\delta \to 0} \frac{\tau_\delta}{\log(\frac{1}{\delta})} \leq (g^*(\boldsymbol{p}))^{-1}.
\end{align*} 
%


\textit{proof of~\eqref{eq:asymp_op_2_iless}}: \\
Let $\zeta>0$. Let $\xi$ be some value less than $\frac{(p_1-p_2)}{4}$. We define $I_{\xi} = [p_1-\xi, p_1+ \xi] \times [p_2 - \xi, p_2+ \xi] \times \ldots \times [p_k-\xi, p_k + \xi].$
Let $T\in \nN$. The event $\mathcal{E}_T$ is defined as 
\begin{equation*}
	\mathcal{E}_T = \cap_{t = T^{\zeta}}^{T}(\boldsymbol{\hat{p}}(t) \in I_{\xi}).
\end{equation*}
From \eqref{NEW_Zabt_RHS}, for the Dirichlet prior, 
\begin{align*}
	Z(t) & = \log((K-1)!) + \sum_{j=1}^{K}\log(N_j(t)!) - \log((t+K-1)!) \nonumber \\
	& \hspace{3cm}- \sum_{i \in \K\setminus\{\hat{a}_t, \hat{b}_t\}}N_i(t)\log(\hat{p}_i(t)) -(N_{\hat{a}_t}(t) + N_{\hat{b}_t}(t))\log\left(\frac{\hat{p}_{\hat{a}_t}(t) + \hat{p}_{\hat{b}_t}(t)}{2}\right).
\end{align*}
By Stirling's approximation, 
\begin{align*}
	N_j(t)! & \sim \sqrt{2\pi N_j(t)}\left(\frac{N_j(t)}{e}\right)^{N_j(t)}\\
	(t+K-1)! & \sim \sqrt{2\pi (t+K-1)}\left(\frac{t+K-1}{e}\right)^{t+K-1}
\end{align*}
as $t \to \infty.$ On $\mathcal{E}_T$, for $t > T^{\zeta}$, $t(p_j-\xi) \leq N_j(t) \leq t(p_j+\xi)$. So, on $\mathcal{E}_T$ there exists a time $T_1>T^{\zeta}$ and some $\lambda >0$ such that for
$t > T_1$,  
\begin{align*}
	N_j(t)!  & \geq   \frac{1}{(1+\lambda)}\sqrt{2\pi N_j(t)}\left(\frac{N_j(t)}{e}\right)^{N_j(t)} \\
	(t+K-1)! & \leq \frac{1}{(1-\lambda)}\sqrt{2\pi (t+K-1)}\left(\frac{t+K-1}{e}\right)^{t+K-1}.
\end{align*}
Using Stirling's approximation in the RHS of above equation of $Z(t)$, on $\mathcal{E}_T$ for $t > T_1$, we get 
\begin{align*}
	& Z(t) \\
	& \geq \log((K-1)!) + \sum_{j=1}^{K}\log\left(\sqrt{2\pi N_j(t)}\right) + \sum_{j=1}^{K}N_j(t)\log(N_j(t)) - t -\log(1+\lambda)\nonumber \\ 
	& \hspace{1cm} -\log\left(\sqrt{2\pi (t + K-1)}\right) - (t + K -1)\log(t + K -1) + t + (K-1) + \log(1-\lambda)\nonumber \\
	& \hspace{2cm} - \sum_{i \in \K\setminus\{1,2\}}N_i(t)\log(\hat{p}_i(t)) -(N_1(t) + N_2(t))\log\left(\frac{\hat{p}_1(t) + \hat{p}_2(t)}{2}\right) \\
	& = \sum_{j=1}^{K}N_j(t)\log\left(\frac{N_j(t)}{t+K-1}\right)  - \sum_{i \in \K\setminus\{a,b\}}N_i(t)\log(\hat{p}_i(t)) -(N_1(t) + N_2(t))\log\left(\frac{\hat{p}_1(t) + \hat{p}_2(t)}{2}\right) \nonumber \\
	& \hspace{1cm} +  \log((K-1)!)  + \sum_{j=1}^{K}\log\left(\sqrt{2\pi N_j(t)}\right) \nonumber \\
	& \hspace{2cm} -\log\left(\sqrt{2\pi (t + K-1)}\right) - ( K -1)\log(t + K -1) + (K-1) + \log\left(\frac{1-\lambda}{1 + \lambda}\right) \nonumber \\
	& \geq \sum_{j=1}^{K}t(p_j-\xi)\log\left(\frac{t(p_j-\xi)}{t+K-1}\right)  - \sum_{i \in \K\setminus\{1,2\}}t(p_j+\xi)\log((p_j+\xi))   \\
	& \hspace{1cm} -t(p_1 + p_2 + 2\xi)\log\left(\frac{p_1 + p_2 + 2\xi}{2}\right) +  \log((K-1)!)  + \sum_{j=1}^{K}\log\left(\sqrt{2\pi t(p_j-\xi)}\right) \nonumber \\
	& \hspace{2cm} -\log\left(\sqrt{2\pi (t + K-1)}\right) - ( K -1)\log(t + K -1) + (K-1) + \log\left(\frac{1-\lambda}{1 + \lambda}\right).
\end{align*}
%
%
On $\mathcal{E}_T$, there exists a time $T_2 >T_1$ such that for $t>T_2$, for some $\varepsilon>0$,
\begin{align*}
	& Z(t) \\
	& > \frac{t}{(1+\varepsilon)}\Bigg(\sum_{j=1}^{K}(p_j-\xi)\log\left(\frac{(p_j-\xi)}{t+K-1}\right)  - \sum_{i \in \K\setminus\{1,2\}}(p_j+\xi)\log((p_j+\xi))   \\
	& \hspace{1cm} -(p_1 + p_2 + 2\xi)\log\left(\frac{p_1 + p_2 + 2\xi}{2}\right) \Bigg) \\
	& = \frac{tf_{\xi}^*(\boldsymbol{p})}{(1+\varepsilon)},
\end{align*}
where $f_{\xi}^*(\boldsymbol{p}):= \sum_{j=1}^{K}(p_j-\xi)\log\left(\frac{(p_j-\xi)}{t+K-1}\right)  - \sum_{i \in \K\setminus\{1,2\}}(p_j+\xi)\log((p_j+\xi)) -(p_1 + p_2 + 2\xi)\log\left(\frac{p_1 + p_2 + 2\xi}{2}\right) .$
Let $T>T_2$. On $\mathcal{E}_T$,
\begin{align*}
	\min(\tau_\delta, T) & \leq T^{\zeta} + \sum_{t=T^{\zeta}}^{T}\mathbb{I}(\tau_\delta >t)\\
	& \leq T^{\zeta}  + \sum_{t=T^{\zeta}}^{T}\mathbb{I}(Z(t) \leq \beta(t, \delta))\\
	& \leq T^{\zeta} + \sum_{t=T^{\zeta}}^{T}\mathbb{I}\left(\frac{tf_{\xi}^*(\boldsymbol{p})}{(1+\varepsilon)} \leq \beta(T, \delta) \right) \\
	& \leq T^{\zeta} + \frac{\beta(T, \delta)(1+\varepsilon)}{f_{\xi}^*(\boldsymbol{p})}.
\end{align*}  
We define,
\begin{align*}
	T_0(\delta) = \inf\left\{T \in \nN: T^\zeta + \frac{\beta(T, \delta)(1+\varepsilon)}{f_{\xi}^*(\boldsymbol{p})} \leq T \right\},
\end{align*}  
such that for every $T \geq \max(T_0(\delta), T_2)$, $\mathcal{E}_T \subseteq ( \tau_\delta \leq T )$. So, for all $T \geq \max(T_0(\delta), T_2)$,
\begin{align}
	\pP(\tau_\delta > T) & \leq \pP(\mathcal{E}_T^c) \nonumber \\
	\text{ and } \eE[\tau_\delta] \leq T_0(\delta) + T_2 + \sum_{T=1}^{\infty}\pP(\mathcal{E}_T^c). \label{new_i:bound_on_E(tau_delta)}
\end{align} 
Let $\eta > 0$. We define a constant
$$C(\eta) = \inf\left\{ T \in \nN: T - T^\zeta \geq \frac{T}{1+\eta} \right\}$$
thus, 
\begin{align}
	T_0(\delta) & \leq C(\eta) + \inf\left\{ T \in \nN: \frac{\beta(T, \delta)(1+\varepsilon)}{f_{\xi}^*(\boldsymbol{p})} \leq \frac{T}{1+\eta} \right\} \nonumber \\
	& = C(\eta) + \frac{\beta(T, \delta)(1+\varepsilon)(1+\eta)}{f_{\xi}^*(\boldsymbol{p})}. \label{new_i:bound_on_T0(delta)}
\end{align}
\begin{lemma}\label{lemma_bound_on_prop_event}
	For some $\zeta>0$,
	\begin{equation}
		\pP(\mathcal{E}_T^c) \leq B\exp(-C(T^{\zeta}+1)) \label{new_i:bound on P(E^c)}
	\end{equation}
	where $B = \frac{2K}{1-\exp(-2\xi^2)}$ and $C = 2\xi^2$.
\end{lemma} 
From \eqref{new_i:bound_on_E(tau_delta)}, \eqref{new_i:bound_on_T0(delta)} and \eqref{new_i:bound on P(E^c)},
\begin{align*}
	\eE[\tau_\delta]  \leq C(\eta) + \frac{\log\left(\frac{K-1}{\delta}\right)(1+\varepsilon)(1+\eta)}{f_{\xi}^*(\boldsymbol{p})} + T_2 + \sum_{T=1}^{\infty}B\exp(-C(T^{\zeta}+1)).
\end{align*}
The upper bound in the RHS of the above equation is finite only if $\sum_{T=1}^{\infty}B\exp(-C(T^{\zeta}+1))$ if finite. Thus, $T^{\zeta}$ is chosen such that $\sum_{T=1}^{\infty}B\exp(-C(T^{\zeta}+1)) < \infty$.
Taking the limit $$\lim_{\xi \to 0}f_{\xi}^*(\boldsymbol{p}) = g^*(\boldsymbol{p})$$ and letting $\varepsilon$ and $\eta$ go to zero we get
\begin{align}
	\lim\sup_{\delta \to 0} \frac{\eE[\tau_\delta]}{\log(\frac{1}{\delta})} \leq (g^*(\boldsymbol{p}))^{-1}.
\end{align}
\textbf{proof of Lemma \ref{lemma_bound_on_prop_event}}
\begin{align}
	\pP(\mathcal{E}_T^c) & \leq \sum_{t=T^{\zeta}}^{T}\pP(\boldsymbol{\hat{p}}(t) \notin I_\xi)\\
	& \leq \sum_{t=T^{\zeta}}^{T} \sum_{a=1}^{K} \left[ \pP(\hat{p}_a(t) < p_a -\xi ) + \pP(\hat{p}_a(t) > p_a +\xi) \right].
\end{align}
Let the random variable 
\begin{equation}
	Z_t^a = 
	\begin{cases}
		1 & \text{ if } X_t =a\\
		0 & \text{otherwise}.
	\end{cases}  
\end{equation}
Since $Z_t^a \leq 1$, it is sub-Gaussian with parameter $\frac{1}{2}$. We can write 
$\hat{p}_a(t) = \frac{1}{t}\sum_{j=1}^{t} Z_j^a$.
By applying the Hoeffding Bound (see~\citet{hoeffding}), 
\begin{equation}
	\pP(\hat{p}_a(t) - p_a > \xi ) \leq \exp(-2t\xi^2) \text{ and } \pP(\hat{p}_a(t) - p_a < -\xi ) \leq \exp(-2t\xi^2).
\end{equation} 
So,
\begin{align}
	\pP(\mathcal{E}_T^c) & \leq 2\sum_{t=T^{\zeta}}^{T} \sum_{a=1}^{K}\exp(-2t\xi^2)\\
	& = 2K \sum_{t=T^{\zeta}}^{T} \exp(-2t\xi^2)\\
	& = 2K \left( \frac{\exp(-2(T^{\zeta}+1)\xi^2) - \exp(-2(T+1)\xi^2)}{1 -\exp(-2\xi^2)}\right)\\
	& \leq 2K\frac{\exp(-2(T^{\zeta}+1)\xi^2)}{1 -\exp(-2\xi^2)}\\
	& = B\exp(-C(T^{\zeta}+1)).
\end{align}
where $B = \frac{2K}{1-\exp(-2\xi^2)}$ and $C = 2\xi^2$.
\end{proof}
\section{Identity-based sampling}
\label{appendix:identitybased}
\begin{proof}[Proof of Lemma~\ref{lemma:lb_identitybased}]
	This proof uses ideas similar to the proof of Theorem 4 in \citet{jain2021sequential}.\\
	Let $D = (d_1, d_2, \ldots d_K) \in \nN^K$ be an instance with community one as the largest community and $N =\sum_{j=1}^{K}d_j$. Let $D' = (d'_1, d'_2,\ldots d'_K)$ be an alternate instance in $\nN^K$ such that the mode is not equal to one and $d'_j \geq d_j$ for all $j \in \mathcal{K}$ and $N' = \sum_{j=1}^{K}d'_j > N$. Let $q_D(s,s')$ and $q_{D'}(s,s')$ denote probability of transition from state $s=(s_1,s_2,\ldots,s_K)$ to $s'=(s'_1, s'_2, \ldots, s'_K)$. So, $q_D(s,s)=\frac{\sum_{j=1}^{K}s_j}{N}$ and $q_D(s,s+e_j) = \frac{d_j - s_j}{N}$, where $e_j$ is a $K$-length unit vector with one in the $j^{th}$ position. Similarly, $q_{D'}(s,s)=\frac{\sum_{j=1}^{K}s_j}{N'}$ and $q_{D'}(s,s+e_j) = \frac{d'_j - s_j}{N'}$. Thus for states $s$ and $s'$ that are feasible under both $D$ and $D'$,
	\begin{equation}
		\log\left(\frac{q_{D}(s,s)}{q_{D'}(s,s)}\right) = \log\left(\frac{N'}{N}\right) \label{trans_same_state}
	\end{equation}
	and  
	\begin{align}
		\log\left(\frac{q_{D}(s,s+e_j)}{q_{D'}(s,s+e_j)}\right) & = \log\left(\frac{(N')(d_j-s_j)}{N(d'_j-s_j)}\right) \nonumber \\
		& \leq \log\left(\frac{N'}{N}\right) \label{trans_diff_state}.
	\end{align}
	Given a state evolution $\left(S(1), S(2), \ldots S(\tau_\delta)\right)$, from Appendix D, proof of Theorem 4 in \citet{jain2021sequential}, the log-likelihood ratio is 
	\begin{equation*}
		\log\frac{\pP_{D}\left(S(1), S(2), \ldots S(\tau_\delta)\right)}{\pP_{D'}\left(S(1), S(2), \ldots S(\tau_\delta)\right)} \leq \sum_{s, s'}N(s, s', \tau_\delta)\log\left(\frac{q_{D}(s,s')}{q_{D'}(s,s')}\right),
	\end{equation*}
	where $N(s, s', \tau_\delta)$ denotes the number of transitions from state $s$ to $s'$ within time $\tau_\delta$. Also, expected value with respect to $D$ of the log-likelihood ratio is
	\begin{align*}
		\eE_{D}\left[\log\frac{\pP_{D}\left(S(1), S(2), \ldots S(\tau_\delta)\right)}{\pP_{D'}\left(S(1), S(2), \ldots S(\tau_\delta)\right)}\right] & \overset{(a)}{\geq} D(Ber(\pP_D(\hat{a}_{\tau_\delta})=1)||Ber(\pP_{D'}(\hat{a}_{\tau_\delta})=1)) \\
		& \overset{(b)}{\geq}kl(\delta,1-\delta) \\
		& \overset{(c)}{\geq} \log\left(\frac{1}{2.4\delta}\right).
	\end{align*}
	In $(a)$, $Ber(x)$ denotes the Bernoulli distribution with parameter $x \in (0,1)$ and this inequality follows from the Data processing inequality(see \citet{cover2012}). The inequalities $(b)$ and $(c)$ follow from section 3.4 in \citet{shah2020sequential}. Therefore, from the above inequality we get 
	\begin{align*}
		\inf_{d'\in D'}\eE_{D}\left[\log\frac{\pP_{D}\left(S(1), S(2), \ldots S(\tau_\delta)\right)}{\pP_{D'}\left(S(1), S(2), \ldots S(\tau_\delta)\right)}\right] \geq \log\left(\frac{1}{2.4\delta}\right) \\
		\inf_{d'\in D'}\sum_{s, s'}\eE_{D}[N(s, s', \tau_\delta)]\log\left(\frac{q_{D}(s,s')}{q_{D'}(s,s')}\right) \geq\log\left(\frac{1}{2.4\delta}\right) \\
		\eE[\tau_\delta]\inf_{d'\in D'}\log\left(\frac{N'}{N}\right) \overset{(f)}{\geq}\log\left(\frac{1}{2.4\delta}\right).
	\end{align*}
	The inequality $(f)$ follows from \eqref{trans_same_state}, \eqref{trans_diff_state} and from the inequality $\sum_{s, s'} N(s, s', t) \leq t$. The above infimum
	is obtained at an alternate instance $D'$ where $d'_j = d_j$ for all $j \neq 2$ and $d_2 =d_1+1$. Thus,
	\begin{align*}
		\eE_D[\tau_\delta] & \geq \frac{\log\left(\frac{1}{2.4\delta}\right)}{\log\left(\frac{N-d_2+d_1+1}{N}\right)}.
	\end{align*}
\end{proof}
%
\begin{proof}[Proof of Lemma~\ref{lemma:comparison_of_LB}]
	Consider the lower bound from Lemma~\ref{lemma:lb_indentity-less} on the expected sample complexity under the identityless sampling model, given by 
	\begin{align}
		(p_1 + p_2)kl\left(\frac{p_1}{p_1+p_2}, \frac{1}{2} \right) & = p_1\log\left(\frac{2p_1}{p_1+p_2}\right) + p_2\log\left(\frac{2p_2}{p_1+p_2}\right) \nonumber \\
		& = (p_1 -p_2)\log\left(\frac{2p_1}{p_1+p_2}\right) + \underbrace{p_2\log\left(\frac{4p_1p_2}{(p_1+p_2)^2}\right)}_{(a)} \nonumber \\
		& \overset{(b)}{\leq} (p_1 - p_2)\log(2) \label{lb_identityless}. 
	\end{align} 
	Inequality $(b)$ follows by noting that the expression in $(a)$ is non-positive since $2\sqrt{p_1p_2} \le (p_1 + p_2)$. Now, consider the lower bound from Lemma~\ref{lemma:lb_identitybased} on the expected sample complexity under the identity-based sampling model, given by 
	\begin{align*}
		\log\left(\frac{N-d_2+d_1+1}{N}\right) & = \log\left(1 + p_1-p_2 + 1/N\right)\\
		& \geq \log(1 + p_1 -p_2)
	\end{align*}
	The function $\log(1 + p_1 -p_2)$ is between $\log(1)$ and $\log(2)$ and for $\alpha = 1-(p_1-p_2)$, as $\log$ is a concave function,
	\begin{align}
		\log\left(\frac{N-d_2+d_1+1}{N}\right) \geq \alpha\log(1) + (1-\alpha)\log(2) = (p_1 -p_2)\log(2) \label{lb_identity}.
	\end{align}
	The result of Lemma~\ref{lemma:comparison_of_LB} follows from \eqref{lb_identityless}, \eqref{lb_identity}. Moreover,
	\begin{align}
		(p_1 + p_2)kl\left(\frac{p_1}{p_1+p_2}, \frac{1}{2} \right) & = p_1\log\left(\frac{2p_1}{p_1+p_2}\right) + p_2\log\left(\frac{2p_2}{p_1+p_2}\right) \nonumber \\
		& \overset{(c)}{<} p_1\left(\frac{2p_1}{p_1+p_2}-1\right) + p_2\left(\frac{2p_2}{p_1+p_2}-1\right) \nonumber \\
		& = \frac{(p_1-p_2)^2}{p_1+p_2}. \label{up_on_lb_iless}
	\end{align}
	The inequality $(c)$ follows as $\log(x) < x-1$ for $x > -1$. Thus, from \eqref{lb_identity} and \eqref{up_on_lb_iless},
	\begin{align}
		\frac{\log\left(\frac{N-d_2+d_1+1}{N}\right)}{(p_1 + p_2 )\ kl\left(
			\frac{p_1}{p_1 + p_2 } , \frac{1}{2} \right)} > \frac{(p_1+p_2)\log(2)}{(p_1-p_2)} \label{eq:lb_ratio}
	\end{align}
\end{proof}
\begin{proof}[Proof of Lemma~\ref{lemma:sup_pf_ST_Yabt_wi}]
	The function $$f_t(\bd') = \left\{\sum_{j \in \mathcal{K}^{(t)}}\int_{-1}^{S_j(t)}
	\log(d'_j-v)dv -t\log\left(\sum_{i=1}^{K}d'_i\right)\right\}.$$
	The KKT conditions to find the supremum of $f_t(\bd')$ under the constraint $d'_a -d'_b \leq 0$ are listed as follows for primal optimal point $\boldsymbol{d}^*$ and dual optimal point $\lambda^*$
	\begin{align*}
		d^*_a & \leq d^*_b\\
		\lambda^* & \geq 0\\
		\lambda^* (d^*_a - d^*_b) & = 0\\
		\nabla f_t(\boldsymbol{d}^*) + \nabla\lambda^* (d^*_a - d^*_b) 	& =0.
	\end{align*}
	So, for $j \in \K \setminus \{a,b\}$,
	\begin{align*}
		\int_{-1}^{S_j(t)}\frac{1}{(d^*_j-y)}dy -\frac{t}{\left(\sum_{i=1}^{K}d^*_i\right)} & =0 \\
		\int_{d^*_j -S_j(t)}^{d^*_j+1}\frac{1}{x}dx -\frac{t}{\left(\sum_{i=1}^{K}d^*_i\right)} & =0\\
		\log\left(\frac{d^*_j+1}{d^*_j -S_j(t)}\right) -\frac{t}{\left(\sum_{i=1}^{K}d^*_i\right)} & =0\\ 
	\end{align*}
	(ii) and for $j=a,b$\\
	\begin{align*}
		\log\left(\frac{d^*_a+1}{d^*_a -S_a(t)}\right) -\frac{t}{\left(\sum_{i=1}^{K}d^*_i\right)} + \lambda & =0\\ 
		\log\left(\frac{d^*_b+1}{d^*_b -S_b(t)}\right) -\frac{t}{\left(\sum_{i=1}^{K}d^*_i\right)} - \lambda & =0. 
	\end{align*}
	From the first three KKT conditions we can see that either $\lambda^* = 0$ or $d^*_a - d^*_b = 0$. If $\lambda^* = 0$  then
	\begin{align*}
		\frac{d^*_j+1}{d^*_j -S_j(t)} & = e^{\frac{t}{\left(\sum_{i=1}^{K}d^*_i\right)}}\\
		e^{\frac{-t}{\left(\sum_{i=1}^{K}d^*_i\right)}}\left(d^*_j+1\right) & = d^*_j -S_j(t) \\
		d^*_j &= \frac{S_j(t) + e^{\frac{-t}{\left(\sum_{i=1}^{K}d^*_i\right)}} }{1-e^{\frac{-t}{\left(\sum_{i=1}^{K}d^*_i\right)}}} \text{ for all } j \in \mathcal{K}.
	\end{align*}
	Thus, when $S_a(t) > S_b(t)$, $d^*_a > d^*_b$. But this violates the constraint. So, when $S_a(t) \geq S_b(t)$,    
	taking $d^*_a = d^*_b$, 
	\\ (i) for $j \in \mathcal{K}^{(t)} \setminus \{a,b\}$
	\begin{align}
		d^*_j = \frac{S_j(t) + e^{\frac{-t}{\left(\sum_{i=1}^{K}d^*_i\right)}} }{1-e^{\frac{-t}{\left(\sum_{i=1}^{K}d^*_i\right)}}} \text{ for } j \in \mathcal{K} \setminus\{a,b\} \label{dj_beta}
	\end{align}
	(ii) and for $j = a,b$
	\begin{align}
		& \log\left(\frac{(d^*_a+1)^2}{(d^*_a -S_a(t))(d^*_a -S_b(t))}\right) -\frac{2t}{\left(\sum_{i=1}^{K}d^*_i\right)}   =0 \label{da_db} \\
		& e^{\frac{-2t}{\left(\sum_{i=1}^{K}d^*_i\right)}}(d^*_a+1)^2  = (d^*_a -S_a(t))(d^*_a -S_b(t)) \nonumber \\
		& d^*_a = d^*_b = \frac{ S_a(t) + S_b(t) + 2e^{\frac{-2t}{\left(\sum_{i=1}^{K}d^*_i\right)}} \pm r\left(e^{\frac{-t}{\left(\sum_{i=1}^{K}d^*_i\right)}}\right)}{2\left(1 - e^{\frac{-2t}{\left(\sum_{i=1}^{K}d^*_i\right)}}\right)} \label{da_db_beta}
	\end{align}
	where $r\left(e^{\frac{-t}{\left(\sum_{i=1}^{K}d'_i\right)}}\right)  = \sqrt{ (S_a(t) - S_b(t))^2 + 4e^{\frac{-2t}{\left(\sum_{i=1}^{K}d'_i\right)}}(1+S_a(t) +S_b(t)+ S_a(t)S_b(t))}$.\\
	\textit{Note: When $S_a(t) < S_b(t)$, by taking $\lambda^* = 0$, $d^*_j = \frac{S_j(t) + e^{\frac{-t}{\left(\sum_{i=1}^{K}d^*_i\right)}} }{1-e^{\frac{-t}{\left(\sum_{i=1}^{K}d^*_i\right)}}}$ for all $j \in \K$, the constraint $d^*_a \leq d^*_b$ is satisfied. Therefore, 
		\begin{equation}\label{unconstrained_sup}
			\text{when $S_a(t) < S_b(t)$,} \sup_{\boldsymbol{d}'\in \R_+^K \colon d'_a \leq d'_b} f_t(\bd')  = \sup_{\boldsymbol{d}'\in \R_+^K}f_t(\bd'), 
		\end{equation}
		the unconstrained supremum of $f_t(\bd')$.}
	
	Let $\gamma = e^{\frac{-t}{\left(\sum_{i=1}^{K}d'_i\right)}}$ and $\gamma_0 = e^{\frac{-t}{\left(\sum_{i=1}^{K}d^*_i\right)}}$. By substituting $\gamma$ in \eqref{dj_beta} and \eqref{da_db_beta} in place of $\gamma_0$, we get the expression of $d'_j$, $d'_a$ and $d'_b$ in terms of $\gamma$ as 
	\begin{align}
		d'_j & = \frac{S_j(t) + \gamma }{1-\gamma} \text{   for $j \neq a,b$ } \text{ and} \label{dj_2_gamma}\\
		d'_a & = d'_b = \frac{ S_a(t) + S_b(t) + 2\gamma^2 \pm r(\gamma)}{2(1 - \gamma^2)}, \label{da_db_2_gamma}
	\end{align} 
	where $r\left(\gamma\right)  = \sqrt{ (S_a(t) - S_b(t))^2 + 4\gamma^2(1+S_a(t) +S_b(t)+ S_a(t)S_b(t))}$. \\
	Therefore, instead of a $K$-dimensional optimization problem we have a one-dimensional optimization problem to find the optimal $\gamma_0$.
	
	\textit{Claim:} $d'_a  = d'_b = \frac{ S_a(t) + S_b(t) + 2\gamma^2 - r(\gamma)}{2(1 - \gamma^2)}$ is not a valid solution. 
	\begin{proof}
		From \eqref{da_db} we can see that the valid value of $d'_a$ is greater than $S_a(t)$ at the critical point. We compare $\frac{ S_a(t) + S_b(t) + 2\gamma^2 - r(\gamma)}{2(1 - \gamma^2)}$ to $S_a(t)$. For 
		\begin{align*}
			& \frac{ S_a(t) + S_b(t) + 2\gamma^2 - r(\gamma)}{2(1 - \gamma^2)}  \leq S_a(t) \\
			& S_a(t) + S_b(t) + 2\gamma^2 - r(\gamma)  \leq  2S_a(t)(1 - \gamma^2) \\
			& S_a(t) + S_b(t) + 2\gamma^2 -2S_a(t) + 2S_a(t)\gamma^2  \leq r(\gamma) \\
			& \left( \left(S_b(t) - S_a(t)\right) + 2\gamma^2(1 + S_a(t))\right)^2 \nonumber \\
			& \hspace{3cm} \leq (S_a(t) - S_b(t))^2 + 4\gamma^2(1+S_a(t) +S_b(t)+ S_a(t)S_b(t)) \\
			& 4\gamma^4(1 + S_a(t))^2  \nonumber \\
			& \hspace{1.5cm} \leq 4\gamma^2( 1+S_a(t) +S_b(t)+ S_a(t)S_b(t) + (S_a(t) - S_b(t))(1 + S_a(t)) ) \\
			& \gamma^2 -1 \leq 0 \\
			& \gamma \leq 1.
		\end{align*}
		From the definition of $\gamma$ it is between $0$ and $1$. From the above equation we can see that $\S_a(t) \geq $\\$\frac{ S_a(t) + S_b(t) + 2\gamma^2 - r(\gamma)}{2(1 - \gamma^2)}$ when $\gamma \leq 1$. So $\frac{ S_a(t) + S_b(t) + 2\gamma^2 - r(\gamma)}{2(1 - \gamma^2)}$ is not a valid solution. 
	\end{proof}
	Let $N^{(a,b)}(\gamma) = \sum_{i=1}^{K}d'_i$ be the sum of $d'_j$s written in terms of $\gamma$ given in \eqref{dj_2_gamma} and \eqref{da_db_2_gamma}.  
	\begin{equation}
		N^{(a,b)}(\gamma) =  \frac{ \sum_{j\in \mathcal{K}^{(t)}\setminus\{a,b\}}S_j(t) + \gamma(K^{(t)}-2) }{1-\gamma} + \frac{ S_a(t) + S_b(t) + 2\gamma^2 + r(\gamma)}{(1 - \gamma^2)} \label{N_gamma}
	\end{equation} 
	Recall that $\gamma = e^{\frac{-t}{\left(\sum_{i=1}^{K}d'_i\right)}} = e^{\frac{-t}{\left(N^{(a,b)}(\gamma)\right)}}$. Combining the expression of $N^{(a,b)}(\gamma)$ and the definition of $\gamma$, the value of $\gamma$ at the critical point is given by,
	\begin{equation}
		\gamma_0  = \exp\left( \frac{-t}{\frac{ \sum_{j\in \mathcal{K}^{(t)}\setminus\{a,b\}}S_j(t) + \gamma_0(K^{(t)}-2) }{1-\gamma_0} + \frac{ S_a(t) + S_b(t) + 2\gamma_0^2 + r(\gamma_0)}{(1 - \gamma_0^2)}}\right). \label{gamma_eq}
	\end{equation}
	From the expression of $\gamma_0$ we get 
	\begin{align*}
		&\log\left(\frac{1}{\gamma_0}\right)  = \frac{t}{\frac{ \sum_{j\in \mathcal{K}^{(t)}\setminus\{a,b\}}S_j(t) + \gamma_0(K^{(t)}-2) }{1-\gamma_0} + \frac{ S_a(t) + S_b(t) + 2\gamma_0^2 + r(\gamma_0)}{(1 - \gamma_0^2)}} \\
		&\log\left(\frac{1}{\gamma_0}\right)\Bigg( \frac{ \sum_{j\in \mathcal{K}^{(t)}\setminus\{a,b\}}S_j(t) + \gamma_0(K^{(t)}-2) }{1-\gamma_0} 
		 + \frac{ S_a(t) + S_b(t) + 2\gamma_0^2 + r(\gamma_0)}{(1 - \gamma_0^2)} \Bigg) -t = g(\gamma_0) =0.
	\end{align*}
	\begin{lemma} \label{lemma_g_gamma}
		$g_1(\gamma) = \log\left(\frac{1}{\gamma}\right)\left( \frac{ \sum_{j\in \mathcal{K}^{(t)}\setminus\{a,b\}}S_j(t) + \gamma(K^{(t)}-2) }{1-\gamma} + \frac{ S_a(t) + S_b(t) + 2\gamma^2 + r(\gamma)}{(1 - \gamma^2)} \right)$ is a decreasing function of $\gamma$.
	\end{lemma}
	\begin{proof}
		Let $h_1(\gamma) = \log\left(\frac{1}{\gamma}\right)\frac{ \sum_{j\in \mathcal{K}^{(t)}\setminus\{a,b\}}S_j(t) + \gamma(K^{(t)}-2) }{1-\gamma}$. 
		The derivative of $h_1(\gamma)$ with respect to $\gamma$ is 
		\begin{align*}
			h'_1(\gamma)  & = \log\left(\frac{1}{\gamma}\right)\frac{ \sum_{j\in \mathcal{K}^{(t)}\setminus\{a,b\}}S_j(t) + \gamma(K^{(t)}-2) }{(1-\gamma)^2} \nonumber \\ 
			& \hspace{1.5cm} + \log\left(\frac{1}{\gamma}\right)\frac{(K^{(t)}-2)}{1-\gamma} - \frac{ \sum_{j\in \mathcal{K}^{(t)}\setminus\{a,b\}}S_j(t) + \gamma(K^{(t)}-2) }{\gamma(1-\gamma)} \\
			h'_1(\gamma)(1-\gamma) & =  \log\left(\frac{1}{\gamma}\right)\frac{ \sum_{j\in \mathcal{K}^{(t)}\setminus\{a,b\}}S_j(t) + \gamma(K^{(t)}-2) }{1-\gamma} \nonumber \\ & \hspace{1.5cm} + \log\left(\frac{1}{\gamma}\right)(K^{(t)}-2) - \frac{ \sum_{j\in \mathcal{K}^{(t)}\setminus\{a,b\}}S_j(t) + \gamma(K^{(t)}-2) }{\gamma}\\
			h'_1(\gamma)(1-\gamma) &  = \sum_{j\in \mathcal{K}^{(t)}\setminus\{a,b\}}S_j(t)\left(\frac{\log\left(\frac{1}{\gamma}\right)}{1-\gamma} -\frac{1}{\gamma}\right) +  (K^{(t)}-2)\left(\frac{\log\left(\frac{1}{\gamma}\right)}{1-\gamma} - 1\right).   
		\end{align*}
		For $h_1(\gamma)$ to be a decreasing function $\gamma$, we want 
		\begin{align*}
			& \sum_{j\in \mathcal{K}^{(t)}\setminus\{a,b\}}S_j(t)\left(\frac{\log\left(\frac{1}{\gamma}\right)}{1-\gamma} -\frac{1}{\gamma}\right) +  (K^{(t)}-2)\left(\frac{\log\left(\frac{1}{\gamma}\right)}{1-\gamma} - 1\right)  \leq 0 \\
			& (K^{(t)}-2)  \leq \sum_{j\in \mathcal{K}^{(t)}\setminus\{a,b\}}S_j(t)\underbrace{\frac{\left(\frac{1}{\gamma} - \frac{\log\left(\frac{1}{\gamma}\right)}{1-\gamma} \right)}{\left(\frac{\log\left(\frac{1}{\gamma}\right)}{1-\gamma} - 1\right)}}_{\nu_\gamma}.
		\end{align*}
		By comparing the ratio $\nu_\gamma$ to one we get 
		\begin{align*}
			& \frac{\left(\frac{1}{\gamma} - \frac{\log\left(\frac{1}{\gamma}\right)}{1-\gamma} \right)}{\left(\frac{\log\left(\frac{1}{\gamma}\right)}{1-\gamma} - 1\right)} \geq  1 \\
			& \frac{1}{\gamma} - \frac{\log\left(\frac{1}{\gamma}\right)}{1-\gamma} \geq \frac{\log\left(\frac{1}{\gamma}\right)}{1-\gamma} - 1 \\
			& \frac{1 + \gamma}{\gamma} \geq 2\frac{\log\left(\frac{1}{\gamma}\right)}{1-\gamma} \\
			& 2\gamma\frac{\log\left(\frac{1}{\gamma}\right)}{1-\gamma^2} \leq 1
		\end{align*}
		So $2\gamma\frac{\log\left(\frac{1}{\gamma}\right)}{1-\gamma^2}$ should be less than one for $\nu_\gamma$ to be geater than one.
		\begin{lemma}\label{lemma_diff}
			The ratio $2\gamma\frac{\log\left(\frac{1}{\gamma}\right)}{1-\gamma^2}$ is less than or equal to 1. 
		\end{lemma}
		\begin{proof}
			Let  $h_0(\gamma) = 1 - \gamma^2 - 2\gamma\log(\frac{1}{\gamma})$. The function
			\begin{equation*}
				h_0(\gamma)  
				\begin{cases}
					= 0 & \text{ at } \gamma =1\\
					\to 1 & \text{as} \gamma \to 0.
				\end{cases}
			\end{equation*}
			The derivative of $h_0(\gamma)$ is 
			\begin{align*}
				\frac{dh_0(\gamma)}{d\gamma} & = -2\gamma - 2\log\left(\frac{1}{\gamma}\right) +2 \\
				& = 2\left( 1 - \gamma - \log\left(\frac{1}{\gamma}\right) \right) \\
				&  = 2( 1 -\gamma + \log(\gamma)) \\
				& \leq 0  \text{ as } \log(\gamma) \leq \gamma -1 \text{ for } \gamma >0.		
			\end{align*}
			Thus $1-\gamma^2 >  2\gamma\log(\frac{1}{\gamma})$.
		\end{proof}
		From the above Lemma we can conclude that $\nu_\gamma \geq 1$. Therefore, $h_1(\gamma)$ is a decreasing function $\gamma$ if $ (K^{(t)}-2)  \leq \sum_{j\in \mathcal{K}^{(t)}\setminus\{a,b\}}S_j(t)$. This is true since $\mathcal{K}^{(t)}$ is the set of integers in $\mathcal{K}$ for which $S_j(t)>0$ and $K^{(t)}$ is the number of communities for which $S_j(t)>0$.\\
		Let $h_2(\gamma) = \log\left(\frac{1}{\gamma}\right)\frac{ S_a(t) + S_b(t) + 2\gamma^2 + r(\gamma)}{(1 - \gamma^2)}$. The derivative of $h_2(\gamma)$ is 
		\begin{align*}
			h'_2(\gamma) & = \frac{\log\left(\frac{1}{\gamma}\right)}{1 - \gamma^2}(4\gamma + r'(\gamma)) + \frac{2\gamma\log\left(\frac{1}{\gamma}\right)}{(1 - \gamma^2)^2}(S_a(t) + S_b(t) + 2\gamma^2 + r(\gamma)) \\
			& \hspace{4cm} -\frac{1}{\gamma(1-\gamma^2)}(S_a(t) + S_b(t) + 2\gamma^2 + r(\gamma)) \\
			h'_2(\gamma)(1 - \gamma^2) & = \underbrace{\frac{4\gamma \log\left(\frac{1}{\gamma}\right)}{1-\gamma^2} -2\gamma}_{pt_1} + \underbrace{\left(\frac{2\gamma\log\left(\frac{1}{\gamma}\right)}{1-\gamma^2} -\frac{1}{\gamma}\right)(S_a(t)+S_b(t))}_{pt_2} \nonumber \\
			& \hspace{1cm} + \underbrace{\left(\frac{2\gamma\log\left(\frac{1}{\gamma}\right)}{1-\gamma^2} -\frac{1}{\gamma}\right)r(\gamma)}_{pt_3} + \underbrace{\log\left(\frac{1}{\gamma}\right)r'(\gamma)}_{pt_4}.
		\end{align*}
		By multiplying the term $(S_a(t) -S_b(t))^2$ by $\gamma^2$ in $r(\gamma)$ we get
		\begin{align}
			r(\gamma)  = \sqrt{ (S_a(t) - S_b(t))^2 + 4\gamma^2(1+S_a(t) +S_b(t)+ S_a(t)S_b(t))}& \nonumber \\
			r(\gamma)  >\sqrt{ (S_a(t) - S_b(t))^2\gamma^2 + 4\gamma^2(1+S_a(t) +S_b(t)+ S_a(t)S_b(t))} & \nonumber \\
			r(\gamma) > \gamma(S_a(t) + S_b(t) + 2) \text{ as the value of $\gamma$ is between $0$ and $1$.} \label{r_gamma_lb} &
		\end{align}
		The derivative of $r(\gamma)$ is
		\begin{align*}
			r'(\gamma) & = \frac{4\gamma(1+S_a(t)+S_b(t) + S_a(t)S_b(t))}{\sqrt{ (S_a(t) - S_b(t))^2 + 4\gamma^2(1+S_a(t) +S_b(t)+ S_a(t)S_b(t))}} \\
			& \leq \frac{4\gamma(S_a(t) + S_b(t))\frac{(1+S_a(t)+S_b(t) + S_a(t)S_b(t))}{(S_a(t) + S_b(t))}}{\sqrt{ ( 4\gamma^2(1+S_a(t) +S_b(t)+ S_a(t)S_b(t))}} \\
			& \leq 2(S_a(t) + S_b(t))\frac{\sqrt{1+S_a(t) +S_b(t)+ S_a(t)S_b(t))}}{(S_a(t) + S_b(t))} \\
			& \leq 2(S_a(t) + S_b(t))
		\end{align*}
		From Lemma \ref{lemma_diff}, we know that $2\gamma\frac{\log\left(\frac{1}{\gamma}\right)}{1-\gamma^2} \leq 1 \leq \frac{1}{\gamma}$. So, using the above bounds on $r(\gamma)$ and $r'(\gamma)$ in the expression of $h'_2(\gamma)(1 - \gamma^2)$ we get
		\begin{align*}
			& h'_2(\gamma)(1 - \gamma^2) \\
			& \leq  \frac{4\gamma \log\left(\frac{1}{\gamma}\right)}{1-\gamma^2} -2\gamma + \left(\frac{2\gamma\log\left(\frac{1}{\gamma}\right)}{1-\gamma^2} -\frac{1}{\gamma}\right)(S_a(t)+S_b(t)) \\
			& \hspace{1cm} + \left(\frac{2\gamma\log\left(\frac{1}{\gamma}\right)}{1-\gamma^2} -\frac{1}{\gamma}\right)(S_a(t)+S_b(t) +2)\gamma  + 2\log\left(\frac{1}{\gamma}\right)(S_a(t) + S_b(t)) \\
			& = \frac{4\gamma(1 +\gamma) \log\left(\frac{1}{\gamma}\right)}{1-\gamma^2} -2(1+\gamma) + (S_a(t) + S_b(t))\left(\frac{2\gamma\log\left(\frac{1}{\gamma}\right)}{1-\gamma^2} -\frac{1}{\gamma}\right)(1 +\gamma) \\ 
			& \hspace{3cm} + 2\log\left(\frac{1}{\gamma}\right)(S_a(t) + S_b(t)) \\
			& =  2(1 + \gamma)\left(\frac{2\gamma\log\left(\frac{1}{\gamma}\right)}{1-\gamma^2} - 1 \right)  + (S_a(t)+ S_b(t))\frac{1+\gamma}{\gamma}\left(\frac{2\gamma\log\left(\frac{1}{\gamma}\right)}{1-\gamma^2} - 1\right)
		\end{align*}
		From Lemma \ref{lemma_diff}, $\frac{2\gamma\log\left(\frac{1}{\gamma}\right)}{1-\gamma^2} \leq 1  $. So the upper bound on $h'_2(\gamma)$ is negative. Thus we can conclude that $h_2(\gamma)$ is a decreasing function of $\gamma$. Since $h_1(\gamma)$ and $h_2(\gamma)$ are decreasing functions of $\gamma$, $g_1(\gamma) = h_1(\gamma) + h_2(\gamma)$ is also a decreasing function of $\gamma$. 
	\end{proof}

	Thus a unique value of $\gamma$ is obtained by equating $g_1(\gamma)$ to $t$. We get the optimal values of $\boldsymbol{d}'$ by substituting $\gamma_0$ in \eqref{dj_2_gamma} and \eqref{da_db_2_gamma}. We will look at the second order conditions to check whether the critical point is indeed a local maximum.
	In order to do so we calculate the following bordered Hessian matrix.
	Let $a = 1$, $b = 2$. 
	The second order derivatives of $f_t(\bd')$ are 
	\begin{align*}
		\frac{\partial^2 f_t(\bd') }{\partial d^{'2}_j} & = \frac{1}{d'_j +1} -\frac{1}{d'_j - S_j(t)} + \frac{t}{\left(\sum_{i=1}^{K}d'_i\right)^2} \\
		& = \frac{-(S_j(t)+1)}{(d'_j+1)(d'_j -S_j(t))} + \frac{t}{\left(\sum_{i=1}^{K}d'_i\right)^2}  \text{ for } j \in \mathcal{K} \\ 
		\frac{\partial^2 f_t(\bd') }{\partial d'_jd'_k} & = \frac{t}{\left(\sum_{i=1}^{K}d'_i\right)^2} \text{ for }  j \neq k \in \mathcal{K} 
	\end{align*}
	Under the constraint $h^{1,2}(\boldsymbol{d}') = d'_1 -d'_2 \leq 0$ the bordered Hessian matrix for the maximization of $f_t(\bd')$ is constructed as
	\begin{align*}
		H_{f_{\boldsymbol{d'}}(t)}	& = \begin{bmatrix}
			0 & \frac{\partial h^{1,2}(\boldsymbol{d}')}{\partial d'_1}  &  \frac{\partial h^{1,2}(\boldsymbol{d}')}{\partial d'_2} & \ldots & \frac{\partial h^{1,2}(\boldsymbol{d}')}{\partial d'_K} \\
			\frac{\partial h^{1,2}(\boldsymbol{d}')}{\partial d'_1} & \frac{\partial^2 f_t(\bd') }{\partial d^{'2}_1} & \frac{\partial^2 f_t(\bd') }{\partial d'_2d'_1} & \ldots & \frac{\partial^2 f_t(\bd') }{\partial d'_kd'_1} \\
			\frac{\partial h^{1,2}(\boldsymbol{d}')}{\partial d'_2} & \frac{\partial^2 f_t(\bd') }{\partial d'_1d'_2} & \frac{\partial^2 f_t(\bd') }{\partial d^{'2}_2} & \ldots & \frac{\partial^2 f_t(\bd') }{\partial d'_kd'_2} \\
			\vdots & \vdots & \vdots & \ddots & \vdots \\
			\frac{\partial h^{1,2}(\boldsymbol{d}')}{\partial d'_K} & \frac{\partial^2 f_t(\bd') }{\partial d'_1d'_k} & \frac{\partial^2 f_t(\bd') }{\partial d'_2d'_K} & \ldots  & \frac{\partial^2 f_t(\bd') }{\partial d^{'2}_k}
		\end{bmatrix} \\
		&  = \frac{t}{\left(\sum_{i=1}^{K}d'_i\right)^2} \begin{bmatrix}
			0 & -R'_0 & R'_0 & 0 & \ldots & 0 \\
			-R'_0 & 1 - R'_1 & 1 & 1 & \ldots & 1 \\ 
			R'_0 & 1 & 1 - R'_2 & 1&  \ldots & 1 \\
			\vdots & \vdots & \vdots & \ddots & \vdots & \vdots \\
			\vdots & \vdots & \vdots & \vdots & \ddots & \vdots \\
			0 & 1 & 1 & 1 & \ldots & 1 - R'_K
		\end{bmatrix} 
	\end{align*}
	where $ R'_0 = \frac{\left(\sum_{i=1}^{K}d'_i\right)^2}{t}$ and $R'_j = \frac{\frac{(S_j(t)+1)}{(d'_j+1)(d'_j -S_j(t))}}{\frac{t}{\left(\sum_{i=1}^{K}d'_i\right)^2}}$.
	\begin{lemma}\label{lemma_prin_minor}
		The last $K-1$ leading principal minors of the Bordered Hessain matrix alternate in sign and the sign of the $l^th$ principal minor is $(-1)^{(l+1)}$. 
	\end{lemma}
	\begin{proof}
		The $l^{th}$ principal minor of $H_{f_{\boldsymbol{d'}}(t)}$ is 
		\begin{equation}
			|M^{p2}_l| = (-1)^{l+1}R_0^{'2}(R'_1 + R'_2)R'_3\ldots R'_{l-1}\left( 1 - \frac{4}{R'_1 + R'_2} - \frac{1}{R'_3} \ldots -\frac{1}{R'_{l-1}}\right) \label{principal_minor}.
		\end{equation}
		At the critical point, for $ j \in \mathcal{K} \setminus \{1,2\}$, $d'_j  = \frac{S_j(t)+\gamma_0}{1 -\gamma_0}$ and $\frac{t}{\left(\sum_{i=1}^{K}d'_i\right)} = \log\left(\frac{1}{\gamma_0}\right)$ (from \eqref{N_gamma} and \eqref{gamma_eq}). So, 
		\begin{align}
			\frac{1}{R'_j} & = \frac{(d'_j+1)(d'_j -S_j(t))t}{(S_j(t)+1)\left(\sum_{i=1}^{K}d'_i\right)^2} \nonumber \\ 
			& = \frac{\gamma_0\log\left(\frac{1}{\gamma_0}\right)(S_j(t)+1)}{\left(\sum_{i=1}^{K}d'_i\right)(1-\gamma_0)^2}.\label{eq_Rj}
		\end{align}
		Let $p = \frac{(d'_1 -S_1(t))}{(d'_1+1)}$ and $q =  \frac{(d'_1 -S_2(t))}{(d'_1+1)}$. For $j =1,2$ at the critical point, from \eqref{da_db} we have $\frac{(d'_1 -S_1(t))(d'_1 -S_2(t))}{(d'_1+1)^2} = pq = \gamma_0^2$. So,  
		\begin{align*}
			& \frac{(S_1(t)+1)}{(d'_1+1)(d'_1 -S_1(t))} + \frac{(S_2(t)+1)}{(d'_1+1)(d'_1 -S_2(t))} \\
			& = \frac{(S_1(t)+1)(d'_1  -S_2(t)) + (S_2(t)+1)(d'_1 -S_1(t))}{(d'_1+1)(d'_1 -S_1(t))(d'_1 -S_2(t))} \\
			& = \frac{(1-p)q + (1-q)p}{(d'_1+1)\gamma_0^2} \\
			& = \frac{(1-p)q + (1-q)p}{(d'_1+1)\gamma_0^2} \\
			& = \frac{\frac{1}{p} + \frac{p}{\gamma_0^2} -2}{(d'_1+1)}.
		\end{align*}
		Let $f(p) = \frac{1}{p} + \frac{p}{\gamma_0^2} -2$. The derivatve of $f(p)$, $f'(p) = \frac{-1}{p^2} + \frac{1}{\gamma_0^2} = 0$ at $p = \gamma_0$ and $f''(p)>0$, so $f(p)$ takes minimum vale at $p =\gamma_0$ and $\min_{p}f(p) = \frac{2(1-\gamma_0)}{\gamma_0}$. Therefore,
		$	\frac{(S_1(t)+1)}{(d'_1+1)(d'_1 -S_1(t))} + \frac{(S_2(t)+1)}{(d'_1+1)(d'_1 -S_2(t))} > \frac{2(1-\gamma_0)}{(d'_1+1)\gamma_0}$. Using this bound on $\frac{4}{R'_1 + R'_2}$ we get 
		\begin{align}
			\frac{4}{R'_1 + R'_2} & = \frac{4\log\left(\frac{1}{\gamma_0}\right)}{\left(\sum_{i=1}^{K}d'_i\right)\left(\frac{(S_1(t)+1)}{(d'_1+1)(d'_1 -S_1(t))} + \frac{(S_2(t)+1)}{(d'_1+1)(d'_1 -S_2(t))}\right)} \nonumber \\
			& \leq \frac{2\gamma_0\log\left(\frac{1}{\gamma_0}\right)(d'_1+1)}{(1-\gamma_0)\left(\sum_{i=1}^{K}d'_i\right)} \nonumber \\
			& = \frac{2\gamma_0\log\left(\frac{1}{\gamma_0}\right)(d'_1+1)}{(1-\gamma_0)\left(\sum_{i=1}^{K}d'_i\right)} \nonumber \\
			& = \frac{\gamma_0\log\left(\frac{1}{\gamma_0}\right) (S_a(t) + S_b(t) + r(\gamma_0) + 2)}{(1-\gamma_0)(1 -\gamma_0^2)\left(\sum_{i=1}^{K}d'_i\right)} \nonumber\\ 
			& = \frac{\gamma_0\log\left(\frac{1}{\gamma_0}\right) (S_a(t) + S_b(t) + r(\gamma_0) + 2)}{(1+\gamma_0)(1 -\gamma_0)^2\left(\sum_{i=1}^{K}d'_i\right)}. \label{eq_R1_R2}
		\end{align}
		Adding  $\frac{1}{R_j}$ for all $j \in \mathcal{K}\setminus \{1,2\}$ and the upper bound on $\frac{4}{R_1 + R_2}$ from equations \eqref{eq_Rj} and \eqref{eq_R1_R2} gives
		\begin{align*}
			& \frac{4}{R'_1 + R'_2} + \sum_{j \in \mathcal{K}\setminus \{1,2\}}\frac{1}{R_j} \nonumber \\
			& \leq \frac{\gamma_0\log\left(\frac{1}{\gamma_0}\right)(S_j(t)+1)}{\left(\sum_{i=1}^{K}d'_i\right)(1-\gamma_0)^2} + \frac{\gamma_0\log\left(\frac{1}{\gamma_0}\right) (S_1(t) + S_2(t) + r(\gamma_0) + 2)}{(1+\gamma_0)(1 -\gamma_0)^2\left(\sum_{i=1}^{K}d'_i\right)}\\
			& =\frac{\gamma_0\log\left(\frac{1}{\gamma_0}\right)}{(1 -\gamma_0)^2\left(\sum_{i=1}^{K}d'_i\right)}\left( \sum_{j \in \mathcal{K}\setminus \{1,2\}}(S_j(t)+1) + \frac{(S_1(t) + S_2(t) + r(\gamma_0) + 2)}{(1+\gamma_0)} \right) \\
			& = \frac{\gamma_0\log\left(\frac{1}{\gamma_0}\right)}{(1 -\gamma_0)^2} \frac{\sum_{j \in \mathcal{K}\setminus \{1,2\}}(S_j(t)+1) + \frac{(S_a(t) + S_b(t) + r(\gamma_0) + 2)}{(1+\gamma_0)}}{\frac{ \sum_{j\in \mathcal{K}^{(t)}\setminus \{1,2\}}S_j(t) + \gamma_0(K^{(t)}-2) }{1-\gamma_0} + \frac{ S_1(t) + S_2(t) + 2\gamma_0^2 + r(\gamma_0)}{(1 - \gamma_0^2)}} \text{ (from \eqref{N_gamma})} \\
			& =  \frac{\gamma_0\log\left(\frac{1}{\gamma_0}\right)}{(1 -\gamma_0)} \frac{ \sum_{j \in \mathcal{K}\setminus \{1,2\}}S_j(t) + \gamma_0(K^{(t)}-2) + (1-\gamma_0)(K^{(t)}-2) + \frac{(S_1(t) + S_2(t) + 2\gamma_0^2 + r(\gamma_0) + 2 -2\gamma_0^2)}{(1+\gamma_0)} }{  \sum_{j\in \mathcal{K}^{(t)}\setminus \{1,2\}}S_j(t) + \gamma_0(K^{(t)}-2)  + \frac{ S_1(t) + S_2(t) + 2\gamma_0^2 + r(\gamma_0)}{(1 + \gamma_0)} }\\
			& = \frac{\gamma_0\log\left(\frac{1}{\gamma_0}\right)}{(1 -\gamma_0)}\left( 1 + \frac{(1-\gamma_0)(K^{(t)}-2) + 2(1-\gamma_0)}{\sum_{j\in \mathcal{K}^{(t)}\setminus \{1,2\}}S_j(t) + \gamma_0(K^{(t)}-2)  + \frac{ S_1(t) + S_2(t) + 2\gamma_0^2 + r(\gamma_0)}{(1 + \gamma_0)}}\right)
		\end{align*}
		Applying the lower bound on $r(\gamma)$ from \eqref{r_gamma_lb} in the above equation and simplifying gives 
		\begin{align*}
			& \frac{4}{R'_1 + R'_2} + \sum_{j \in \mathcal{K}\setminus \{1,2\}}\frac{1}{R_j} \nonumber \\ 
			& \leq \frac{\gamma_0\log\left(\frac{1}{\gamma_0}\right)}{(1 -\gamma_0)}\left(1 + \frac{K^{(t)}(1-\gamma_0)}{\sum_{j\in \mathcal{K}^{(t)}\setminus \{1,2\}}S_j(t) + \gamma_0(K^{(t)}-2)  + \frac{ S_1(t) + S_2(t) + 2\gamma_0^2 + \gamma_0(S_1(t)+S_2(t)+2)}{(1 + \gamma_0)}}\right)\\
			& =  \frac{\gamma_0\log\left(\frac{1}{\gamma_0}\right)}{(1 -\gamma_0)}\left(1 + \frac{K^{(t)}(1-\gamma_0)}{\sum_{j\in K^{(t)}}S_j(t) + K^{(t)}\gamma_0}\right)\\
			& <   \frac{\gamma_0\log\left(\frac{1}{\gamma_0}\right)}{(1 -\gamma_0)}\left(1 + \frac{K^{(t)}(1-\gamma_0)}{K^{(t)}(1 +\gamma_0)}\right) \\
			& = \frac{2\gamma_0\log\left(\frac{1}{\gamma_0}\right)}{(1 -\gamma_0^2)} < 1 \text{ (from lemma \ref{lemma_diff})}.
		\end{align*}
		So, from \eqref{principal_minor} and the bound above we can conclude that the last (K-1) leading principal minors of the bordered Hessian matrix alternate in sign and the sign of the $l^{th}$ principal minor is $(-1)^{l+1}$. This is a sufficient condition for the critical point to be local maximum. This lemma is proved for $a =1, b=2$. It can similarly be proved for any $a \in \mathcal{K}$ and $b \in \mathcal{K}\setminus\{a\}$.
	\end{proof}
	From lemma \ref{lemma_prin_minor} we can say that the critical point is a local maximum. Since there is a unique value of $\gamma$ for which $g(\gamma) = g_1(\gamma) - t$, the critical point is indeed a global maximum.
	When $\gamma \to 1$, the ratio $\frac{\log\left(\frac{1}{\gamma}\right)}{1-\gamma}$ tends to 1. So,
	\begin{align}\label{eq:g_1_gamma_as_gamma_to_1}
		g_1(\gamma) & = \frac{\log\left(\frac{1}{\gamma}\right)}{1-\gamma} \left( \sum_{j \in \K^{(t)} \setminus \{a,b\}}S_j(t) + \gamma (K^{(t)}-1) + \frac{S_a(t) + S_b(t) + 2\gamma^2 + r(\gamma)}{1+\gamma} \right) \nonumber \\
		& \overset{(u)}{\to} \sum_{j \in \K}S_j(t) +K^{(t)} \text{ as } \gamma \to 1.
	\end{align} 
	The limit $(u)$ holds because from definition of $r(\gamma)$, $r(1) = S_a(t) + S_b(t) + 2$ and  $\sum_{j \in \K^{(t)}} S_j(t)= \sum_{j \in \K} S_j(t)$.
	Thus, the zero of $g(\gamma)$ for $t = \sum_{j \in \K} S_j(t) + K^{(t)}$ is obtained at $\gamma \to 1$. This implies that the values of $d'_j$ for all $j \in \K$ tend to infinity. Thus the identity based rule applies for $t > \sum_{j \in \K}S_j(t) +K^{(t)}.$
\end{proof}
\begin{proof}[Proof of Lemma~\ref{Lemma:Y}]
	Let $t > \sum_{j\in \K}S_j(t) + K^{(t)}$. When $S_a(t) > S_b(t)$, from~\eqref{eq:Wab_T2}, Lemma~\ref{lemma:sup_pf_ST_Yabt_wi} and \eqref{Yabt_lower_bound}
\begin{align}
Y_{a,b}(t)  =   &\log\left( \sum_{\boldsymbol{d}' \in \S_{\alpha}(t)} \frac{\prod_{j=1}^{K}\prod_{l=0}^{S_j(t)-1}(d'_j-l)\pP_{\bd'\sim \D}(\boldsymbol{d}')}{(\sum_{i=1}^{K}d'_i)^t} \right)  - \sup_{\boldsymbol{d}'\in \R_+^K \colon d'_a \leq d'_b} f_t(\bd'), \nonumber \\
\end{align}
	Whereas from~\eqref{unconstrained_sup}, 
	\begin{align}
		Y_{b,a}(t)  =   &\log\left( \sum_{\boldsymbol{d}' \in \S_{\alpha}(t)} \frac{\prod_{j=1}^{K}\prod_{l=0}^{S_j(t)-1}(d'_j-l)\pP_{\bd'\sim \D}(\boldsymbol{d}')}{(\sum_{i=1}^{K}d'_i)^t} \right)-\sup_{\boldsymbol{d}'\in \R_+^K}f_t(\bd'). \label{Ybat}
	\end{align}
	Comparing $Y_{a,b}(t)$ with $Y_{b,a}(t)$ we get
	\begin{equation}
		Y_{a,b}(t) >  Y_{b,a}(t) \text{ when }  S_a(t) \geq S_b(t) \label{Yab_vs_Yba}.
	\end{equation}
	Also, note that for any pair $a,b$, when $S_a(t) \geq S_b(t)$, $Y _{b,a}(t)$ has the same value given in \eqref{Ybat}. Following arguments similar to those in 
	Lemma \ref{Lemma:Z},
	\begin{align*}
		Y(t)  = \max_{a \in \K}\min_{b \in \K\setminus\{a\}}Y_{a,b}(t) = Y_{\tilde{a}_t, \tilde{b}_t}(t).
	\end{align*}
\end{proof}
%

\begin{proof}[Proof of Theorem~\ref{thm:correctness_identitybased}]
	If the mode is chosen by \nialgo, the mode is correct with probability atleast $1 - \delta$. This is evident from the proof of Theorem \ref{thm:correctness_identityless}. Now, we prove the correctness of mode chosen by the identity-based stopping rule when condition \eqref{eq:Yabt_stopping_rule} is satisfied.
	Let $a^*$ denote the mode. For $b \in \K$, let  
	$$\mathit{E}_b^t = \{\omega: \tau_\delta =t, a^* = b\}$$ 
	be the set of sample paths for which the stopping time of identity-based stopping rule is $t$ and the mode estimated is $b$. 
	Note that 
	$$\mathit{E}_b^{t_1}\cap\mathit{E}_b^{t_2} = \phi \text{ for all } t_1 \neq t_2.$$
	The probability of error in estimating the mode
	\begin{align*}
		\pP(a^* \neq 1) & = \sum_{b \in \mathcal{K}\setminus\{1\}} \sum_{t=0}^{\infty} \sum_{\omega \in \mathit{E}_b^t} \mathcal{L}_{\boldsymbol{d}}(\omega) \\
		& \leq  \sum_{b \in \mathcal{K}\setminus\{1\}} \sum_{t=0}^{\infty} \sum_{\omega \in \mathit{E}_b^t} \sup_{d'_1 \geq d'_b}\mathcal{L}_{\boldsymbol{d'}}(\omega) \\
	\end{align*}
	%
	%
	\begin{align*}
		\pP(a^* \neq 1) & \leq \sum_{b \in \mathcal{K}\setminus\{1\}} \sum_{t=0}^{\infty} \sum_{\omega \in \mathit{E}_b^t} \frac{\sup_{d'_1 \geq d'_b}\mathcal{L}_{\boldsymbol{d'}}(\omega)\mathbb{E}_{\bd'\sim \D}\left[\mathcal{L}_{\boldsymbol{d}'}(\bx^t, \bsigma^t)\right]}{\mathbb{E}_{\bd'\sim \D}\left[\mathcal{L}_{\boldsymbol{d}'}(\bx^t, \bsigma^t) \right]} \\
		& \overset{(1)}{\leq} \sum_{b \in \mathcal{K}\setminus\{1\}} \sum_{t=0}^{\infty} \sum_{\omega \in \mathit{E}_b^t} e^{-W_{b,1}(t)}\mathbb{E}_{\bd'\sim \D}\left[\mathcal{L}_{\boldsymbol{d}'}(\bx^t, \bsigma^t) \right] \\
		& \overset{(2)}{\leq}\sum_{b \in \mathcal{K}\setminus\{1\}} \sum_{t=0}^{\infty} \sum_{\omega \in \mathit{E}_b^t} e^{-Y_{b,1}(t)}\mathbb{E}_{\bd'\sim \D}\left[\mathcal{L}_{\boldsymbol{d}'}(\bx^t, \bsigma^t) \right].
	\end{align*}
	The inequality $(1)$ follows from the definition of $W_{a,b}(t)$ in \eqref{def_Wabt} and $(2)$ follows from \eqref{Yabt_lower_bound}.
	Since $Y_{b,1}(t)$ is greater than threshold $\beta(t,\delta/2)$ at stopping time $t$,
	\begin{align*}
		\pP(a^* \neq 1) 
		& \leq \sum_{b \in \mathcal{K}\setminus\{1\}} e^{-\beta(t,\delta/2)} \underbrace{\sum_{t=0}^{\infty} \sum_{\omega \in \mathit{E}_b^t}\mathbb{E}_{\bd'\sim \D}\left[\mathcal{L}_{\boldsymbol{d}'}(\bx^t, \bsigma^t)\right]}_{(p*)} \\
		& \leq \frac{\delta}{2},
	\end{align*}
	for $\beta(t, \delta/2) = \log\left(\frac{2(K-1)}{\delta}\right)$, as $(p*)$ is $\pP( a^* = b)$ under an alternate probability distribution.
\end{proof}



\begin{proof}[Proof of Theorem~\ref{thm:optimality_identitybased}]
	\textit{Proof of \eqref{eq:asymp_op_1_wi}}:\\
	Let event 
	$$\mathcal{E}_1 = \{ \forall j \in \mathcal{K}, S_j(t) \xrightarrow[t \to \infty]{} d_j \}.$$
	\begin{lemma}
		The event $\mathcal{E}_1$ is of probability 1.
	\end{lemma}
	\begin{proof}
		Let $A_{j,t}$ be the event that $j^{\text{th}}$ individual is not picked in $t$ samples. Event $A_{j,t+1} \subset A_{j,t}$. The event that  $j^{\text{th}}$ individual is not picked as $t \to \infty$ can be written as $\cap_{t=1}^{\infty}A_{j,t}$.
		The probablity 
		\begin{align*}
			\pP(\mathcal{E}_1^c) & \leq \pP(\cup_{j=1}^{N}\cap_{t=1}^{\infty}A_{j,t}) \\ 
			& \leq \sum_{j=1}^{N}\pP(\cap_{t=1}^{\infty}A_{j,t})\\
			& = \sum_{j=1}^{N} \lim_{t\to \infty }\pP(A_{j,t})\\ 
			& = \sum_{j=1}^{N}\lim_{t\to \infty }\left(\frac{N-1}{N}\right)^t\\
			& =0.
		\end{align*}
		So, the event $\mathcal{E}_1$ is of probability 1 as $t \to \infty$.  
	\end{proof}
    On $\mathcal{E}_1$, there exists a time $t_1 > N$ such that for $t>t_1$, $S_j(t) = d_j$. For $t > t_1$,
	\begin{align*}
		 Y(t) & = \log\left( \sum_{\boldsymbol{d}'\in \nN^K, d_j\leq d'_j\leq\alpha d_j}\frac{\prod_{j=1}^{K}\prod_{l=0}^{d_j-1}(d'_j-l)\pP_{\boldsymbol{d}'\sim \D}(\boldsymbol{d}')}{(\sum_{i=1}^{K}d'_i)^t} \right)  \\
		& \hspace{4cm}  -\left( \sum_{j \in \mathcal{K}^{(t)}}\int_{-1}^{d_j}\log(d^{1, 2}_j(\gamma_0)-v)dv -t\log\left(\sum_{i=1}^{K}d^{1,2}_i(\gamma_0)\right) \right).
	\end{align*}
	The summation 
	\begin{align*}
		& \sum_{\boldsymbol{d}'\in \nN^K, d_j\leq d'_j\leq\alpha d_j}\frac{\prod_{j=1}^{K}\prod_{l=0}^{d_j-1}(d'_j-l)\pP_{\boldsymbol{d}'\sim \D}(\boldsymbol{d}')}{(\sum_{i=1}^{K}d'_i)^t} \\
		& > \sum_{\boldsymbol{d}'\in \nN^K, d_j\leq d'_j\leq\alpha d_j}\frac{\pP_{\boldsymbol{d}'\sim \D}(\boldsymbol{d}')}{(\sum_{i=1}^{K}d'_i)^t} \\
		& = \frac{\pP_{\boldsymbol{d}'\sim \D}(\boldsymbol{d}'=\bd)}{N^t} + \sum_{\boldsymbol{d}'\in \nN^K, d_j< d'_j\leq\alpha d_j}\frac{\pP_{\boldsymbol{d}'\sim \D}(\boldsymbol{d}')}{(\sum_{i=1}^{K}d'_i)^t}\\
		& > \frac{\pP_{\boldsymbol{d}'\sim \D}(\boldsymbol{d}'=\bd)}{N^t}.
	\end{align*}
	So, for $t > t_1$,
	\begin{align*}
		 Y(t) 
		& > \log\left(  \pP_{\boldsymbol{d}'\sim \D}(\boldsymbol{d}'=\bd)\right) -t\log(N) \nonumber \\
		& \hspace{4cm}  - \sum_{j \in \mathcal{K}^{(t)}}\int_{-1}^{d_j}\log(d^{1, 2}_j(\gamma_0)-v)dv + t\log\left(\sum_{i=1}^{K}d^{1,2}_i(\gamma_0)\right).
	\end{align*}
	From lemma \ref{lemma_g_gamma}, as $g_1(\gamma)$ is a deceasing function of $\gamma$, for $t \to \infty$, the zero of $g(\gamma) = g_1(\gamma) - t$ is achieved at $\gamma \to 0.$ As $\gamma_0 \to 0$, $d^{1,2}_1(\gamma_0)$ and $d^{1, 2}_2(\gamma_0)$ tend to $S_{1}(t)$ and $d^{1,2}_j (\gamma_0)\to S_j(t)$. 
	Since as $t \to \infty$, $S_j(t) \to d_j$ for all $j \in \mathcal{K}$ and $Y(t) \to t\log(\frac{N-d_2 + d_1}{N})$, there exists a time $t_2 > t_1 $
	such that for $t>t_2$ and some $\epsilon>0$,
	\begin{equation*}
		Y(t) >\frac{t}{1+\epsilon}\log\left(\frac{N-d_2 + d_1}{N}\right).
	\end{equation*}
	The stopping time 
	\begin{align*}
		\tau_\delta &= \inf\left\{t\in\nN:  Y(t) > \log\left(\frac{K-1}{\delta}\right)\right\} \\
		& \leq t_2\vee \inf\left\{ t \in \nN: t(1 + \epsilon)^{-1}\log\left(\frac{N-d_2 + d_1}{N}\right) \geq \log\left( \frac{(K-1)}{\delta} \right) \right\}.
	\end{align*}
	Hence,
	\begin{equation*}
		\tau_\delta \leq t_2 \vee (1+\epsilon)\log\left( \frac{(K-1)}{\delta} \right)\log\left(\frac{N-d_2 + d_1}{N}\right)^{-1}.
	\end{equation*}
	Thus, on $\mathcal{E}_1$, for every $\delta \in (0,1]$, $\tau_\delta$ is finite and 
	\begin{align}\label{eq:wi_op}
		\lim\sup_{\delta \to 0} \frac{\tau_\delta}{\log(\frac{1}{\delta})} \leq \left(\log\left(\frac{N-d_2 + d_1}{N}\right)\right)^{-1}
	\end{align} 
	when $\epsilon \to 0$.
	For the \nialgo\ optimality is proved in Theorem \ref{thm:optimality_identityless}. From Lemma~\ref{lemma:comparison_of_LB}, $\left(\log\left(\frac{N-d_2 + d_1}{N}\right)\right)^{-1} < (g^*(\boldsymbol{p}))^{-1}$. So, the theorem statement is implied by \eqref{eq:wi_op}.


\textit{Proof of \eqref{eq:asymp_op_2_wi}}:\\
	We define an event  

	\begin{align}
		\mathcal{E}_T = \cap_{t =\sqrt{T}}^{T}(S(t) = \boldsymbol{d}).
	\end{align}
	On $\mathcal{E}_T$, for $t > \sqrt{T}$,
	\begin{align*}
		Y(t) \nonumber 
		& = \log\left( \sum_{\boldsymbol{d}'\in \nN^K, d_j\leq d'_j\leq\alpha d_j}\frac{\prod_{j=1}^{K}\prod_{l=0}^{d_j-1}(d'_j-l)\pP_{\boldsymbol{d}'\sim \mathcal{D}}(\boldsymbol{d}')}{(\sum_{i=1}^{K}d'_i)^t} \right)  \\
		& \hspace{4cm}  -\left( \sum_{j \in \mathcal{K}^{(t)}}\int_{-1}^{d_j}\log(d^{1,2}_j(\gamma_0)-v)dv -t\log\left(\sum_{i=1}^{K}d^{1,2}_i(\gamma_0)\right) \right).
	\end{align*}
	Taking $\alpha = 1$, on $\mathcal{E}_T$, for $t > \sqrt{T}$,

	\begin{align*}
		Y(t) \nonumber 
		& > \log\left(\prod_{j=1}^{K}\prod_{l=0}^{d_j-1}(d_j-l)\pP_{\boldsymbol{d}'\sim \mathcal{D}}(\boldsymbol{d}'=\bd) \right) -t\log(N) \\
		& \hspace{4cm}  -\left( \sum_{j \in \mathcal{K}^{(t)}}\int_{-1}^{d_j}\log(d^{1,2}_j(\gamma_0)-v)dv -t\log\left(\sum_{i=1}^{K}d^{1,2}_i(\gamma_0)\right) \right).
	\end{align*}
	As $t \to \infty$, $\gamma_0 \to 0$. Consequently, $d^{1,2}_1(\gamma_0)$ and $d^{1,2}_2(\gamma_0)$ tend to $S_a(t)$ and $d^{1,2}_j(\gamma_0) \to S_j(t)$ and $\sum_{i=1}^{K}d^{1,2}_i(\gamma_0) \to (N - d_2 + d_1)$ as $t \to \infty$.
	So, on $\mathcal{E}_T$, there exists a time $T_1>\sqrt{T}$ such that for $t > T_1$, 
	\begin{align*}
		Y(t) \nonumber 
		& > \frac{t}{1 + \varepsilon}\log\left(\frac{N - d_2 + d_1}{N}\right).
	\end{align*}
	Let $T \geq T_1$. On $\mathcal{E}_T$,
	\begin{align}
		\min(\tau_\delta, T) & \leq \sqrt{T} + \sum_{t=\sqrt{T}}^{T}\mathbb{I}(\tau_\delta >t)\\
		& \leq \sqrt{T} + \sum_{t=\sqrt{T}}^{T}\mathbb{I}(Y(t) \leq \beta(t, \delta/2))\\
		& \leq \sqrt{T} + \sum_{t=\sqrt{T}}^{T}\mathbb{I}\left(\frac{t}{1 + \varepsilon}\log\left(\frac{N - d_2 + d_1}{N}\right)  \leq \beta(T, \delta/2) \right) \\
		& \leq \sqrt{T} + \frac{\beta(T, \delta/2)(1 + \varepsilon)}{\log\left(\frac{N - d_2 + d_1}{N}\right)}
	\end{align}  
	we define 
	\begin{align*}
		T_0(\delta) = \inf\left\{T \in \nN: \sqrt{T} + \frac{\beta(T, \delta/2)(1 + \varepsilon)}{\log\left(\frac{N - d_2 + d_1}{N}\right)} \leq T \right\},
	\end{align*}  
	such that for every $T \geq \max(T_0(\delta), T_1)$, $\mathcal{E}_T \subseteq ( \tau_\delta \leq T )$. So, for all $T \geq \max(T_0(\delta), T_1)$,
	\begin{align}
		\pP(\tau_\delta > T) & \leq \pP(\mathcal{E}_T^c) \nonumber \\
		\text{ and } \eE[\tau_\delta] \leq T_0(\delta) + T_1 + \sum_{T=1}^{\infty}\pP(\mathcal{E}_T^c). \label{wi:bound_on_E(tau_delta)}
	\end{align} 
	Let $\eta > 0$. We define a constant
	$$C(\eta) = \inf\left\{ T \in \nN: T - \sqrt{T} \geq \frac{T}{1+\eta} \right\}$$
	thus, 
	\begin{align}
		T_0(\delta) & \leq C(\eta) + \inf\left\{ T \in \nN: \frac{\beta(T, \delta/2)(1 + \varepsilon)}{\log\left(\frac{N - d_2 + d_1}{N}\right)} \leq \frac{T}{1+\eta} \right\} \nonumber \\
		& = C(\eta) + \frac{\beta(T, \delta/2)(1 + \varepsilon)(1 + \eta)}{\log\left(\frac{N - d_2 + d_1}{N}\right)}. \label{wi:bound_on_T0(delta)}
	\end{align}
	Let $A_{j,t}$ be the event that $j^{th}$ individual is not picked in $t$ samples. Event $A_{j,t+1} \subset A_{j,t}$. So, 
	\begin{align*}
		\pP(\mathcal{E}_T^c) & \leq \pP(\cup_{j=1}^{N}A_{t,\sqrt{T}}) \\ 
		& \leq \sum_{j=1}^{N}\pP(A_{j,\sqrt{T}})\\ 
		& = \sum_{j=1}^{N}\left(\frac{N-1}{N}\right)^{\sqrt{T}}\\
		& = N\left(\frac{N-1}{N}\right)^{\sqrt{T}}.
	\end{align*}
	Then, 
	\begin{align}
		\sum_{T=1}^{\infty}\pP(\mathcal{E}_T^c) & \leq N\sum_{T=1}^{\infty}\left(\frac{N-1}{N}\right)^{\sqrt{T}} \nonumber \\
		& \leq N\int_{T=1}^{\infty}\left(\frac{N-1}{N}\right)^{\sqrt{T-1}} \nonumber \\
		& = \frac{2N}{\log^2\left(\frac{N-1}{N}\right)} \label{wi:bound on P(E^c)}
	\end{align}
	From \eqref{wi:bound_on_E(tau_delta)}, \eqref{wi:bound_on_T0(delta)} and \eqref{wi:bound on P(E^c)},
	\begin{align}
		\eE[\tau_\delta] \leq C(\eta) + \frac{\log\left(\frac{2*(K-1)}{\delta}\right)(1 + \varepsilon)(1 + \eta)}{\log\left(\frac{N - d_2 + d_1}{N}\right)} + T_1 + \frac{2N}{\log^2\left(\frac{N-1}{N}\right)}.
	\end{align}
	So, for any $\varepsilon >0$ and $\eta > 0$, 
	\begin{align*}
		\limsup_{\delta \to 0}\frac{\eE[\tau_\delta]}{\log\left(\frac{1}{\delta}\right)} \leq \frac{(1 + \varepsilon)(1 + \eta)}{\log\left(\frac{N - d_2 + d_1}{N}\right)}.
	\end{align*}
	As $\varepsilon$ and $\eta$ go to zero,
	\begin{align*}
		\limsup_{\delta \to 0}\frac{\eE[\tau_\delta]}{\log\left(\frac{1}{\delta}\right)} \leq \frac{1}{\log\left(\frac{N - d_2 + d_1}{N}\right)}.
	\end{align*}
	In Theorem \ref{thm:optimality_identityless} optimality of \nialgo\ is proved. From Lemma~\ref{lemma:comparison_of_LB}, $\left(\log\left(\frac{N-d_2 + d_1}{N}\right)\right)^{-1} < (g^*(\boldsymbol{p}))^{-1}$. So, the theorem statement is implied by the above equation.
\end{proof}

\end{document}